\documentclass[12pt]{amsart}

\usepackage{amsfonts,amssymb,enumerate,bm,hhline,makecell,array,mathtools}
\usepackage{mathrsfs}
\usepackage{comment}
\usepackage[normalem]{ulem}
\usepackage[colorlinks=true,citecolor=blue, urlcolor=blue, linkcolor=blue, 
]{hyperref}
\usepackage{amscd}
\usepackage[shortlabels]{enumitem}
\setlist[itemize]{noitemsep,nolistsep,topsep=-3pt}
\setlist[enumerate]{noitemsep,nolistsep,topsep=-3pt}
\usepackage{tensor}
\usepackage{tikz}
\usetikzlibrary{matrix,arrows}
\usepackage{extarrows}
\usepackage[toc,page]{appendix}
\usepackage[all,ps,cmtip,rotate]{xy}
\usepackage{color}

\allowdisplaybreaks

\def\cI{\mathscr{I}}

\def\cL{\mathscr{L}}
\def\cO{\mathscr{O}}

\def\cP{\mathscr{P}}

\def\cE{\mathscr{E}}

\def\cS{\mathscr{S}}
\def\cC{\mathscr{C}}
\def\cM{\mathscr{M}}

\def\cK{\mathscr{K}}

\def\cQ{\mathscr{Q}}

\def\cX{\mathscr{X}}

\def\Z{{\bf Z}}

\def\C{{\bf C}}

\def\R{{\bf R}}
\def\Q{{\bf Q}}
\def\P{{\bf P}}
\def\PP{{\bf P}}

\def\phi{\varphi}

\DeclareMathOperator{\Alb}{Alb}

\DeclareMathOperator{\Eff}{Eff}

\DeclareMathOperator{\Gr}{Gr}
\DeclareMathOperator{\Hdg}{Hdg}

\def\Im{\mathop{\rm Im}\nolimits}

\DeclareMathOperator{\Ker}{Ker}

\DeclareMathOperator{\Mov}{\overline{Mov}}

\DeclareMathOperator{\Nef}{Nef}

\DeclareMathOperator{\NS}{NS}
\DeclareMathOperator{\num}{\underset{\mathrm num}{\equiv}}

\def\otp{}

\DeclareMathOperator{\Pos}{\overline{Pos}}
\DeclareMathOperator{\Psef}{Psef}

\DeclareMathOperator{\rk}{rk}

\DeclareMathOperator{\Sym}{Sym}

\DeclareMathOperator{\td}{td}

\def\ot{\otimes}
\def\op{\oplus}

\def\isom{\simeq}
\def\cong{\isom}

\def\lra{\longrightarrow}
\def\llra{\hbox to 10mm{\rightarrowfill}}
\def\lllra{\hbox to 15mm{\rightarrowfill}}

\def\llla{\hbox to 10mm{\leftarrowfill}}
\def\lllla{\hbox to 15mm{\leftarrowfill}}

\def\dra{\dashrightarrow}

\def\lhra{\ensuremath{\lhook\joinrel\relbar\joinrel\to}}

\DeclareMathOperator{\isomto}{\stackrel{{}_{\scriptstyle\sim}}{\to}}

\DeclareMathOperator{\isomdra}{\stackrel{{}_{\scriptstyle\sim}}{\dra}}
\DeclareMathOperator{\isomlra}{\stackrel{{}_{\scriptstyle\sim}}{\lra}}

\def\HK{Hyper-K\"ahler}

\def\hk{hyper-K\"ahler}

\newtheorem{lemm}{Lemma}[section]
\newtheorem{theo}[lemm]{Theorem}
\newtheorem{coro}[lemm]{Corollary}
\newtheorem{prop}[lemm]{Proposition}
\newtheorem{claim}[lemm]{Claim}

\theoremstyle{remark}

\newtheorem{rema}[lemm]{Remark}
\newtheorem{conj}[lemm]{Conjecture}

\newtheorem{exam}[lemm]{Example}

\makeatletter
\@namedef{subjclassname@2020}{%
  \textup{2020} Mathematics Subject Classification}
\makeatother

\makeatletter
\def\@tocline#1#2#3#4#5#6#7{\relax
  \ifnum #1>\c@tocdepth 
  \else
    \par \addpenalty\@secpenalty\addvspace{#2}%
    \begingroup \hyphenpenalty\@M
    \@ifempty{#4}{%
      \@tempdima\csname r@tocindent\number#1\endcsname\relax
    }{%
      \@tempdima#4\relax
    }%
    \parindent\z@ \leftskip#3\relax \advance\leftskip\@tempdima\relax
    \rightskip\@pnumwidth plus4em \parfillskip-\@pnumwidth
    #5\leavevmode\hskip-\@tempdima
      \ifcase #1
       \or\or \hskip 1em \or \hskip 2em \else \hskip 3em \fi%
      #6\nobreak\relax
    \dotfill\hbox to\@pnumwidth{\@tocpagenum{#7}}\par
    \nobreak
    \endgroup
  \fi}
\makeatother

\def\hh{\ensuremath{\mathsf h}}
\def\lll{\ensuremath{\mathsf l}}
\def\eee{\ensuremath{\mathsf e}}
\def\fff{\ensuremath{\mathsf f}}
\def\mm{\ensuremath{\mathsf m}}
\def\nn{\ensuremath{\mathsf n}}

\def\M{\ensuremath{\mathsf M}}
\def\K{\ensuremath{\mathsf K}}

\title[Hyper-K\"ahler fourfolds]{Computing Riemann--Roch polynomials and classifying  hyper-K\"ahler fourfolds}

\begin{document}

\author[O.~Debarre]{Olivier Debarre}
\address{Universit\'{e}  Paris Cit\'e and  Sorbonne Universit\'{e}, CNRS, IMJ-PRG, F-75013 Paris, France\vspace{-1mm}}
\email{{olivier.debarre@imj-prg.fr}}

\author[D.~Huybrechts]{Daniel Huybrechts}
\address{Universit\"at Bonn, Mathematisches Institut and Hausdorff Center for Mathematics, Endenicher Allee 60, 53115 Bonn, Germany\vspace{-1mm}}
\email{{huybrech@math.uni-bonn.de}}

\author[E.~Macr\`i]{Emanuele Macr\`i}
\address{Universit\'e Paris-Saclay, CNRS, Laboratoire de Math\'ematiques d'Orsay, Rue Michel Magat, B\^at. 307, 91405 Orsay, France\vspace{-1mm}}
\email{{emanuele.macri@universite-paris-saclay.fr}}

\author[C.~Voisin]{Claire Voisin}
\address{Sorbonne Universit\'{e} and Universit\'{e}  Paris Cit\'e, CNRS, IMJ-PRG, F-75005 Paris, France\vspace{-1mm}}
\email{{claire.voisin@imj-prg.fr}}

\thanks{This project has received funding from the European
Research Council (ERC) under the European
Union's Horizon 2020 research and innovation
programme (ERC-2020-SyG-854361-HyperK)}

\begin{abstract}
We prove that a hyper-K\"ahler fourfold satisfying a mild topological assumption is of K3$^{[2]}$ deformation type.\
This proves in particular a conjecture of O'Grady stating that hyper-K\"ahler fourfolds of K3$^{[2]}$ numerical type are of K3$^{[2]}$ deformation type.\
Our topological assumption concerns the existence of two integral degree-2 cohomology classes satisfying certain numerical intersection conditions.

There are two main ingredients in the proof.\
We first prove a topological version of the statement, by showing that our topological assumption forces the Betti numbers, the Fujiki constant, and the Huybrechts--Riemann--Roch polynomial of the hyper-K\"ahler fourfold to be the same as those of K3$^{[2]}$ hyper-K\"ahler fourfolds.\
The key part of the article is then to prove the hyper-K\"ahler SYZ conjecture for hyper-K\"ahler fourfolds for divisor classes satisfying the numerical condition mentioned above.
\end{abstract}

\maketitle

\tableofcontents
\setcounter{tocdepth}{1}

\section{Introduction}\label{sec:intro}

We work over the complex numbers.\
A \hk\ manifold is a simply connected smooth compact K\"ahler manifold which carries a nowhere degenerate holomorphic 2-form (called a symplectic form) unique up to multiplication by a nonzero  constant.\
This important class of manifolds is a generalization in all even dimensions of K3 surfaces.

There are two key conjectures about \hk\ manifolds.

\begin{conj}\label{conj:defoSYZ}
Any \hk\ manifold can be deformed into a \hk\ manifold with a Lagrangian fibration.
\end{conj}

Let $2n$ be the dimension of $X$ and assume that there exists a nonzero class $\lll\in H^2(X,\Z)$ such that $\int_X\lll^{2n}=0$.\
One can deform $X$ so that $\lll$ becomes of type~$(1,1)$, hence is the first Chern class of some nontrivial holomorphic line bundle $L$ on $X$.\
For a very general deformation of this kind, $X$ has Picard number one.\
In this case, its K\"ahler cone coincides with its positive cone (\cite{HuyHK}, \cite[Proposition 3.2]{huyKahler}) and, therefore, $L$ or its dual is nef.\
Conjecture~\ref{conj:defoSYZ} is then implied by the following.

\begin{conj}[\hk\ SYZ conjecture]\label{conj:SYZ}
Let $X$ be a \hk\ manifold of dimension~$2n$.\
Any nontrivial nef line bundle $L$ on $X$ such that $\int_X c_1(L)^{2n}=0$ is semi-ample.
\end{conj}

The existence of a class $\lll$ as above is equivalent to the existence of a nonzero class in~$H^2(X,\Q)$ that is isotropic with respect to the Beauville--Bogomolov form $q_X$ (see equation~\eqref{fuj}).\
By Meyer's theorem, such isotropic classes exist when \mbox{$b_2(X)\geq 5$.}\

A line bundle $L$ as in Conjecture~\ref{conj:SYZ} would then define a fibration $f\colon X\to B$, with~$B$ normal.\
By results of Matsushita, $B$ is projective, $f$ is a Lagrangian hence equidimensional fibration, and any smooth fiber $X_b\coloneqq f^{-1}(b)$ is an abelian variety, endowed  with a canonical polarization which is the positive generator of the saturation of the image of the rank-1 restriction map $H^2(X,\Z)\to H^2(X_b,\Z)$ (see Section~\ref{sec:LagrangianFibration}).\ 

It is conjectured that the base $B$ of any Lagrangian fibration is  smooth, and in fact $\P^n$; note that $B$ is smooth if and only if $f$ is flat.\ 
When~$X$ is projective and $B$ is smooth, it was proved by Hwang in \cite{hw} that $B$ is $\P^n$; this was extended to the non-projective case by Greb and Lehn in \cite{gl}.\ When~$X$ is projective and~$n=2$, results of Ou in \cite{ou} and  Huybrechts and Xu  in \cite{hx} imply that $B$ is always $\P^2$.\ 

Both conjectures are true  for all known examples of \hk\ manifolds: for Conjecture~\ref{conj:defoSYZ}, see Section~\ref{subsec:SYZKnownExs}; the following stronger form of  Conjecture~\ref{conj:SYZ} follows from work of Matsushita and others.

\begin{theo}\label{thm:SYZKnownExs}
Let $X$ be a \hk\ manifold of dimension $2n$ of either $\mathrm{K3}^{[n]}$,  generalized Kummer,  $\mathrm{OG10}$, or $\mathrm{OG6}$ deformation type.\
Let $L$ be a   nef line bundle on $X$ such that $\int_X c_1(L)^{2n}=0$ and $c_1(L)$ is primitive in $H^2(X,\Z)$.\
There exists a Lagrangian fibration $f\colon X\to \P^n$ such that $f^*\cO_{\P^n}(1)\cong L$.
\end{theo}

Indeed, a ``rational version'' of Conjecture~\ref{conj:SYZ}, namely \cite[Conjecture 1.1]{mat:isotropic}, is known for a  movable line bundle $L$ with primitive isotropic first Chern class by \cite[Corollary 1.1]{mat:isotropic} in the case of $\mathrm{K3}^{[n]}$ or generalized Kummer deformation type (based on \cite{bm,marlag,yos}), by \cite[Theorem 2.2]{mo} in the case of OG10 deformation type, and by \cite[Theorem 7.2]{mora} in the case of OG6 deformation type.\
When $L$ is nef,  we can apply \cite[Claim 3.1 and Claim 3.2]{mat:isotropic} to conclude.

One of our main results is that both conjectures are true in dimension 4 under an additional topological assumption that we now explain.\ We introduce another class $\mm\in H^2(X,\Z)$ with $q_X(\lll,\mm)>0$.\ 
The number
\begin{equation}\label{defaintro}
    a\coloneqq\frac{1}{n!}\int_X\lll^n\mm^n
\end{equation}
is then a positive integer (see Lemma~\ref{lemmaa}).\ 
For most of this article, {\em we make the assumption $a=1$, that is, $\int_X\lll^n\mm^n=n!$} (the minimal possible value).

Assume  $\mm$ is the first Chern class of  a line bundle $M$ on $X$.\
When there is a Lagrangian fibration $f\colon X\to \P^n$ and $L=f^*\cO_{\P^n}(1)$, the restriction of $M$ to any smooth fiber is  a polarization of ``degree''~$a$.\
So the condition $a=1$ means that these  polarizations are  principal.\
This set up was the starting point of this work.\
We prove in Section~\ref{sec:LagrangianFibration}  that upon replacing $\mm$ by $\mm+r\lll$, for a suitable integer~$r$,  we can assume $q_X(\lll,\mm)=1$ and $q_X(\mm)=0$.\
By mimicking the work \cite{ort} of R\'ios Ortiz, we find that the  Huybrechts--Riemann--Roch polynomial~$P_{RR,X}$ (see  Section~\ref{subsec:RRpolynomial} for the definition) then takes the very simple form
\begin{equation}\label{eqRRintro}
\forall k\in\Z\qquad P_{RR,X}(2k)=P_{RR,X}(q_X(k\lll+\mm))=\chi(X,L^k\ot M)=\chi(\P^n,\cO_{\P^n}(k+1))
\end{equation}
(see Theorem~\ref{th21} for more properties of~$X$).

We are naturally led to asking whether this conclusion still holds without assuming the existence of the Lagrangian fibration~$f$.

\begin{conj}\label{conjnum}
Let $X$ be a \hk\ manifold of dimension $2n$ with  classes $\lll,\mm\in H^2(X,\Z)$ such that
\[
\int_X\lll^{2n}=0\quad\textnormal{and}\quad\int_X\lll^n\mm^n=n!.
\]
Then the Huybrechts--Riemann--Roch polynomial of $X$  satisfies
$$\forall k\in\Z\qquad P_{RR,X}(2k)=\chi(\P^n,\cO_{\P^n}(k+1))=\binom{k+1+n}{n}.$$
\end{conj}

Conjecture \ref{conjnum} is not strictly speaking implied by Theorem~\ref{th21} and the SYZ conjecture.\ Indeed, the latter predicts that if $L$ is a nef holomorphic line bundle, {\it some positive power} of~$L$ is generated by global sections and gives a Lagrangian fibration.\
The relation~\eqref{eqRRintro} proves Conjecture~\ref{conjnum}  when $\lll$ is the first Chern class of a nef holomorphic line bundle $L$ such that~$L$ itself is generated by global sections.\
The question of whether $L$ is generated by global sections, if a power induces a Lagrangian fibration, was recently studied  in~\cite{kv}, where the authors give  a sufficient condition for this to happen; unfortunately, their result does not apply in our situation.

Our other results exclusively deal with the case of dimension $4$.\
In that dimension, there are two known deformation classes of \hk\ manifolds, namely the deformations of the Hilbert square $S^{[2]}$ of  a K3 surface $S$ (called the $\mathrm{K3}^{[2]}$ deformation type) and that of the generalized Kummer variety $\K_2(A)$ of an abelian surface $A$.\

We say that a \hk\ fourfold $X$ is \emph{of $\mathrm{K3}^{[2]}$ numerical type} if, for some K3 surface~$S$, there exists an isomorphism of abelian groups $\psi\colon H^2(X,\Z)\isomto H^2(S^{[2]},\Z)$ such that, for all $\alpha\in H^2(X,\Z)$, we have $\int_X \alpha^4 = \int_{S^{[2]}} \psi(\alpha)^4$ (this is equivalent to asking that the lattices $(H^2(X,\Z),q_X) $ and $(H^2(S^{[2]},\Z),q_{S^{[2]}}) $ be isometric and the Fujiki constants be the same).\

In \cite{og1}, O'Grady  conjectured  that a \hk\ fourfold   of  K3$^{[2]}$ numerical type is of~K3$^{[2]}$ deformation type and provided strong evidence for this statement.\
One of our results implies O'Grady's conjecture under a much weaker  topological  hypothesis.

\begin{theo}\label{thm:MainThm}
Let $X$ be a \hk\ fourfold.\
Assume there are classes $\lll$, $\mm\in H^2(X,{\Z})$ such that $\int_X \lll^4=0$ and $\int_X\lll^2\mm^2=2$.\
Then $X$ is of~$\mathrm{K3}^{[2]}$ deformation type.
\end{theo}

\begin{coro}[O'Grady's conjecture]\label{cor:OGrady}
A \hk\ fourfold   of  $\mathrm{K3}^{[2]}$ numerical type  is of~$\mathrm{K3}^{[2]}$ deformation type.
\end{coro}

Let us explain  how we prove Theorem~\ref{thm:MainThm}.\
The first step   is to show Conjecture~\ref{conjnum} in dimension $4$ (see Theorem~\ref{th43a}). 

\begin{theo}\label{thm:RRpoly}
Let $X$ be a \hk\ fourfold with classes $\lll,\mm\in H^2(X,\Z)$ such that $\int_X \lll^4=0$ and $\int_X\lll^2\mm^2=2$.\
The Huybrechts--Riemann--Roch polynomial of $X$ satisfies
\[
P_{RR,X}(2k)=\chi(\P^2,\cO_{\P^2}(k+1))=\binom{k+3}{2}.
\]
Furthermore, the Fujiki constant and the Hodge and Chern numbers of $X$ are those of the Hilbert square of a K3 surface.
\end{theo}

The conclusion  of Theorem \ref{thm:RRpoly} is  weaker than being of ${\mathrm K3}^{[2]}$ numerical type.\
Theorem~\ref{thm:RRpoly} is the starting point for the  second step of the proof, in which we establish
 Conjecture~\ref{conj:SYZ} in dimension $4$  under the same  assumptions on $X$ made in Theorem~\ref{thm:MainThm}.\
More precisely, our key result is the following.

\begin{theo}\label{thm:SYZ}
Let $X$ be a \hk\ fourfold and let $L$ be  a nef line bundle on $X$.\ Set
$\lll\coloneqq c_1(L) \in H^2(X,\Z)$ and assume that there exists $\mm\in H^2(X,\Z)$ such that     $\int_X \lll^4=0$ and $\int_X\lll^2\mm^2=2$.\
There exists a Lagrangian fibration $f\colon X\to\P^2$ with $f^*{\cO_{\P^2}}(1)\isom L$.
\end{theo}
As we mentioned earlier, \hk\ manifolds  of $\mathrm{K3}^{[n]}$ deformation type satisfy the SYZ conjecture, so Theorem \ref{thm:SYZ} is weaker than Theorem~\ref{thm:MainThm}.\  In our approach, Theorem \ref{thm:SYZ} (or rather  some weaker versions of it) is a step towards
proving   Theorem~\ref{thm:MainThm}, the precise logical relationships being   as follows.\
We first consider the case  of a very general \hk\ fourfold~$X$   satisfying the   assumptions  of Theorem \ref{thm:SYZ} and for which the class $\mm$ is also of type~$(1,1)$, that is, $\mm=c_1(M)$ for some line bundle $M$ on $X$.\
In such a case, we can assume that $\mm$ also satisfies $\int_X \mm^4=0$ and is in the boundary of the positive cone of~$X$.

A study of the effective cone of $X$ then shows that there are two cases.
\begin{itemize}
    \item The first case, mostly studied in Section \ref{sec:TwoNefIsotropic}, is when $L$ and $M$ are both nef and $L\ot M$ is ample.\
We show in  Proposition~\ref{prop:CaseC1} that in this case,
\begin{itemize}
    \item[$\circ$] either, after possibly permuting $L$ and $M$, Theorem~\ref{thm:SYZ} holds:   the linear system~$|L|$ induces a Lagrangian fibration to $\P^2$.\ But this contradicts what we had proved earlier in Section~\ref{subsec:ExceptionalDivisor}: that, if $|L|$ induces a Lagrangian fibration, $L$ and $M$ cannot be both nef;
    \item[$\circ$] or any divisor in the linear system $|L\otimes M|$ is irreducible and  the image of the rational map $\phi_{L\otimes M}\colon X\dra \P^5$ is rationally connected.\ But this contradicts   \cite[Theorem 4.2]{voisinfib}   showing that this situation cannot happen, at least when $X$ is very general as above  with Picard number $2$.
\end{itemize}
So this case in fact  does not arise (Corollary~\ref{cor:NoTwoIsotropicNef}).
    \item The second case, mostly studied in Section \ref{sec:DivisorialContraction}, is when $X$ admits a divisorial contraction and $M$ is not nef.\
    In this case, we  first prove Proposition~\ref{prop:CaseC2} (a slightly weaker version of Theorem~\ref{thm:SYZ}), namely the existence of a Lagrangian fibration $f\colon X\to\P^2$ with $f^*{\cO_{\P^2}}(1)\isom L^{k_L}$, for some positive integer~$k_L$.\
    We then use this weaker result to prove that $X$ is of  K3$^{[2]}$ deformation type (see Section~\ref{sec:OGrady}).
\end{itemize}

This completes the proof of Theorem~\ref{thm:MainThm}: as explained at the beginning of Section~\ref{sec:OGrady}, by deformation, one can always assume that the classes $\lll$ and $\mm$ are of type $(1,1)$ and the triple $(X,\lll,\mm)$ satisfies the hypotheses made above.\
But, by Theorem~\ref{thm:SYZKnownExs}, this also proves  Theorem~\ref{thm:SYZ}, once we know (by Theorem~\ref{thm:MainThm}) that $X$ is of  K3$^{[2]}$ deformation type.\ Note that our results rely, in the first case described above, on the classification work of \cite[Section 4]{voisinfib}, which extends to our context O'Grady's analysis in \cite{og1}.

We end the article with Section~\ref{sec:further}, which includes boundedness results when we fix the dimension $2n$ and the integer $a$ defined in~\eqref{defaintro}, and an (incomplete) analysis  of what happens when $n=2$ and $a$ is small.

\subsection*{Acknowledgements}
We would like to thank Enrico Arbarello, Ciro Ciliberto, Giovanni Mongardi, Kieran O'Grady, Gianluca Pacienza, \'Angel David R\'ios Ortiz, Giulia Sacc\`a, Jieao Song, and Chenyang Xu for useful discussions, suggestions, and references.\ We also thank the referee for useful comments.


\section{Review of hyper-K\"ahler manifolds}\label{sec:RRpolynomial}

We recall in Section~\ref{subsec:BBF} the definitions of  Beauville--Bogomolov--Fujiki forms and  Fujiki constants for \hk\ manifolds.\
We then define  their Huybrechts--Riemann--Roch polynomials  and   review some of their elementary properties   in Section~\ref{subsec:RRpolynomial}.

\subsection{The Beauville--Bogomolov--Fujiki form and the Fujiki constant}\label{subsec:BBF}

Let $X$ be a \hk\ manifold of dimension $2n$.\
There exists a canonical integral nondivisible quadratic form $q_X$ (the \emph{Beauville--Bogomolov--Fujiki form}) on $H^2(X,\Z)$ and a positive rational constant $c_X$ (the \emph{Fujiki constant}) such that
\begin{equation}\label{fuj}
\forall \alpha\in H^2(X,\R)\qquad \int_X \alpha^{2n}=c_X\, q_X(\alpha)^n
\end{equation}
(after extending $q_X$ to a quadratic form on $H^2(X,\R)$).\
Moreover, $q_X(h)>0$ for all K\"ahler  classes~$h$ (see \cite[Proposition~23.14]{huy}).

Assume $q_X(\alpha)=0$.\
Then $\alpha^{n+1}=0$ (see for example \cite[Proposition~24.1]{huy}) and, by comparing the coefficients of $t^n$ in the relation
\[
\int_X (t\alpha+ \beta)^{2n}=c_Xq_X(t\alpha+ \beta)^n=c_X(2tq_X( \alpha, \beta)+q_X( \beta))^n,
\]
we get
\begin{equation}\label{pol}
\forall \beta\in H^2(X,\R)\qquad \frac{1}{2^n}\binom{2n}{n}\int_X \alpha^{n}\beta^{n}= c_Xq_X( \alpha, \beta)^n.
\end{equation}
This is a particular case of the polarization of the Fujiki relation~\eqref{fuj}, which in dimension 4 takes the form
 \begin{equation}\label{eq:Fujiki3}
 3\!\int_X \alpha_1\alpha_2\alpha_3\alpha_4 =c_X\bigl( q_X(\alpha_1,\alpha_2)q_X(\alpha_3,\alpha_4) + q_X(\alpha_1,\alpha_3)q_X(\alpha_2,\alpha_4) + q_X(\alpha_1,\alpha_4)q_X(\alpha_2,\alpha_3)\bigr)
\end{equation}
for all $\alpha_1, \alpha_2, \alpha_3,\alpha_4\in H^2(X,\Z)$.\

\subsection{The  Huybrechts--Riemann--Roch polynomial}\label{subsec:RRpolynomial}

By \cite[Corollary~23.18]{huy}, there is a polynomial
\begin{equation*}
P_{RR,X}(T)=\sum_{i=0}^n  a_i T^i\in\Q[T]
\end{equation*}
of degree $n$  such that, for each line bundle $L$ on $X$, one has
\begin{equation}\label{rr}
\chi(X,L)=P_{RR,X}(q_X(c_1(L))).
\end{equation}
This polynomial has the following properties:
\begin{enumerate}[{\rm (a)}]
\item\label{ita} the  constant   term of $P_{RR,X}(T)$ is $\chi(X,\cO_X)=n+1$;
\item\label{itb} the leading term of $P_{RR,X}(T)$ is $\frac{c_X}{(2n)!} T^n $;
\item the coefficients of $P_{RR,X}(T)$ are all positive (\cite[Theorem~1.1]{jia}).
\end{enumerate}

The next two lemmas are elementary.

\begin{lemm}\label{lem11}
Let $X$ be a \hk\ manifold.\ For every class $\alpha\in H^2(X,\Z)$, one has $P_{RR,X}(q_X(\alpha))\in\Z$.
\end{lemm}

\begin{proof}
Since the period map is surjective (\cite[Proposition~25.12]{huy}), the class $\alpha$ becomes the first Chern class of some holomorphic line bundle on a deformation $X'$ of $X$, and~\eqref{rr} implies $P_{RR,X'}(q_{X'}(\alpha))\in \Z$.\ Since the polynomial~$P_{RR,X} $ and the form $q_X$ are deformation invariant, the lemma follows.
\end{proof}

\begin{lemm}\label{lemmaa}
Let $X$ be a \hk\ manifold.\ Assume   that there is a  class~$\lll \in H^2(X,\Z)$ such that $\int_X \lll^{2n}=0$.\ For every $ \mm\in H^2(X,\Z)$, the number
\begin{equation}\label{defa2}
a\coloneqq \frac{1}{n!}\int_X \lll^{n}\mm^{n}
\end{equation}
is  an integer.
\end{lemm}

\begin{proof}
By~\eqref{fuj}, we have $q_X(\lll)=0$.\ Set
\begin{equation}\label{defP}
\forall k\in \Z\qquad P(k)\coloneqq P_{RR,X}(q_X(k\lll+ \mm))=P_{RR,X}(2kq_X(\lll,\mm)+q_X(\mm)) .
\end{equation}
Then $P$ is a  polynomial of degree $n$ whose leading coefficient is  (use~\eqref{pol} and property~\ref{itb} above)
\begin{equation}\label{defia}
\frac{c_X}{(2n)!} (2 q_X(\lll,\mm))^n=\frac{1}{(n!)^2}\int_X \lll^{n}\mm^{n}=\frac{a}{n!}.
\end{equation}
By Lemma~\ref{lem11}, the polynomial $P$ takes integral values on integers, hence $a$ is  an integer.
\end{proof}

\begin{rema}\label{prim}
Under the hypotheses of Lemma~\ref{lemmaa}, when $a$ is    divisible by no nontrivial $n$th powers, the sublattice $\Z\lll\oplus \Z\mm$ of $H^2(X,\Z)$ is saturated: this follows from the fact that  $\frac{1}{n!}\int_X \lll^{n}\alpha^{n}$ is an integer for all $\alpha\in H^2(X,\Z)$.
\end{rema}


\section{Lagrangian fibrations}\label{sec:LagrangianFibration}

Let $X$ be again a \hk\ manifold of dimension $2n$.\
We assume in this section that there is a  fibration $f\colon X\to \P^n$.\
Set $L\coloneqq f^*\cO_{\P^n}(1)$, with first Chern class $\lll\in H^2(X,\Z)$; it satisfies \mbox{$q_X(\lll)=0$.}\
Let $\mm\in H^2(X,\Z)$ be another class, not necessarily of type $(1,1)$; we assume \mbox{$q_X(  \lll, \mm)>0$.}

Any smooth fiber $X_b\coloneqq f^{-1}(b)$ is a Lagrangian complex torus  of dimension~$n$ (\cite[Theorem~1]{mat}, \cite[Theorem 1]{ac}).\
By \cite[Lemma~2.2]{matlag}, the hyperplane $\lll^\bot$ is contained in the kernel of the  restriction map $r_b\colon H^2(X,\C)\to H^2(X_b,\C)$.\
Since the restriction of a K\"ahler class on $X$ is a K\"ahler class on $X_b$, the map $r_b$ has rank exactly~$1$ and the rational class $r_b(\mm)$ is  a positive multiple of a K\"ahler class, hence is an ample class on $X_b$.\
In particular, $X_b$ is an abelian variety (\cite[Proposition~4]{ac}) and the ``degree'' of the polarization $r_b(\mm)=\mm\vert_{X_b}$ is the positive integer
\begin{equation*}
  \frac{1}{n!}\int_{X_b}(\mm\vert_{X_b})^{n}= \frac{1}{n!}\int_X \lll^{n}\mm^{n}=a
\end{equation*}
already considered in~\eqref{defa2}.

There are restrictions on the values that $a$ can take (see Theorem~\ref{th43}   for restrictions on small values of $a$ when $n=2$).\
For the moment, we prove that the existence of the Lagrangian fibration $f$ imposes strong conditions on~$X$ when~$a=1$, that is, when there is a class on $X$ that induces a principal polarization on the smooth fibers of $f$.

\subsection{The Huybrechts--Riemann--Roch polynomial}\label{subsec:RRpolynomialLagrangian}

Keeping the notation and the hypotheses as above, we show that in the case $a=1$, the Huybrechts--Riemann--Roch polynomial of $X$ is completely determined.\
The main idea of the proof is taken from~\cite{ort}: the hypothesis~$a=1$ is essential because it implies that a certain locally free sheaf of rank $a$ on~$\P^n$ can be written as $\cO_{\P^n}(d)$ for some integer $d$.

\begin{theo}\label{th21}
Let $X$ be a \hk\ manifold of dimension~$2n$ with a Lagrangian fibration $f\colon X\to \P^n$.\
Set $\lll\coloneqq c_1(f^*\cO_{\P^n}(1))\in H^2(X,\Z)$ and assume that there exists $\mm\in H^2(X,\Z)$ such that $\int_X \lll^{n}\mm^{n}=n!$.\
Then, $q_X(\lll,\mm)=\pm 1$, the quadratic form~$q_X$ is even, $c_X=(2n-1)!!$,
\begin{equation*}
P_{RR,X}(T)= \binom{\frac{T}{2}+n+1}{n},
\end{equation*}
and the sublattice $\Z\lll\oplus \Z\mm$ of $(H^2(X,\Z),q_X)$ is isomorphic to a hyperbolic plane.
\end{theo}

For all known \hk\ manifolds $X$, the lattice $(H^2(X,\Z),q_X)$ contains a hyperbolic plane.

\begin{proof}
Changing $\mm$ into $-\mm$ is necessary, we may assume $q_X(\lll,\mm)>0$.\
Consider the universal family $(\cX,\cL)\to {\mathfrak M}_\lll$ over one component of the (non-Hausdorff) moduli space of marked hyper-K\"ahler manifolds with a fixed $(1,1)$-class $\lll$.\
The period map $\cP\colon{\mathfrak M}_\lll\to \PP(\lll^\perp\otimes\C)$ is injective over very general points and the points in an arbitrary fiber correspond to the chambers of the decomposition of the positive cone.\
For a distinguished point $0\in {\mathfrak M}_\lll$ we have $(\cX_0,\cL_0)\cong(X,L )$, where $L\coloneqq f^*\cO_{\P^n}(1)$.\
But there also exists a fiber $X^\prime\coloneqq\cX_{0^\prime}$, with $0^\prime\in{\mathfrak M}_\lll$, whose rational N\'eron--Severi group~$\NS(X^\prime)_\Q$ is generated by the classes $\lll$ and $\mm$.\
In fact, since the lattice generated by $\lll$ and~$\mm$ is saturated (Remark~\ref{prim}), the integral N\'eron--Severi group  is generated by  $\lll$ and $\mm$.\
Furthermore, replacing $0^\prime$ by another point in the same fiber over $\cP(0^\prime)$, we may assume that $L^\prime\coloneqq\cL_{0^\prime}$ is nef.\
Since $q_X(k\lll+\mm)=2kq_X( \lll,\mm)+q_X( \mm)>0$ for $k\gg 0$, the manifold  $X'$ is projective by \cite{HuyHK}.

We apply~\cite[Theorem 1.2, Claim 3.1 and Claim 3.2]{mat:isotropic}: there exists a Lagrangian fibration $f^\prime\colon X^\prime\to \P^n$ such that $f^{\prime\,*}\cO_{\P^n}(1)\cong L^\prime$.\
Upon replacing $(X,L)$ by $(X^\prime,L^\prime)$, we may, since $q_X$, $c_X$, and $P_{RR,X}(T)$ are invariant by deformation, assume that $X$ carries a line bundle~$M$ with first Chern class $\mm$.

Since $L$ is nef, $q_X(\lll,\mm)>0$, and $\NS(X)$ is generated by $\lll$ and $\mm$, we can replace $M$ with $L^k\otimes M$, for $k\gg0$, and assume that $M$ is ample on $X$.\  By \cite[Theorem~10.32]{kol}, $R^if_*M$ vanishes for $i>0$.\ Because $f$ is flat,     the sheaf $\cM\coloneqq f_*M$ on~$\P^n$ is locally free; since
the restriction of $M$ to a smooth generic fiber of $f$ defines a principal polarization, the rank of $\cM$ is  $a=1$.\
By the projection formula and~\eqref{rr}, $\cM$ satisfies
\begin{equation}\label{chi}
\forall k\in\Z\qquad\chi(\P^n, \cM(k))=\chi(X,  L^k\otimes M) =P_{RR,X}(2kq_X(\lll,\mm)+q_X(\mm)) .
\end{equation}
Following \cite[Section~3]{ort}, we write $\cM=\cO_{\P^n}(d)$ for some integer $d$ and, from~\eqref{chi}, we deduce
\[
\forall k\in\Z\qquad P_{RR,X}(2kq_X(\lll,\mm)+q_X(\mm))=\chi(\P^n,\cO_{\P^n}(d+k))=\binom{d+k+n}{n},
\]
so that
\begin{equation}\label{bin}
P_{RR,X}(T)= \binom{d+\frac{T-q_X(\mm)}{2 q_X(\lll,\mm)}+n}{n}.
\end{equation}
Set $\gamma\coloneqq \frac{ q_X(\mm)}{2 q_X(\lll,\mm)} $.\ Since $P_{RR,X}(0)=n+1$, either $d-\gamma=1$, or $n$ is even and $d-\gamma=-n-2$.\footnote{The equation $\binom{x+n}{n}= n+1$ is equivalent to the monic equation $\prod_{i=1}^n(x+i)= (n+1)! $, so any rational solution is in fact integral, and the only integral solutions are $x=1$ and, when $n$ is even, $x=-n-2 $.}\
Since the coefficients of $P_{RR,X}$ are positive and the coefficient of $T^{n-1}$ in~\eqref{bin} is a positive multiple of $ n+2(d-\gamma) +1$, the latter case is ruled out.\
So $\gamma$ is an integer and $d=\gamma+1$.\ Replacing $M$ by $L^{-\gamma}\otimes M$, we may  assume $\gamma=0$ and $d=1$, hence $q_X(\mm)=0$, so that~\eqref{bin} becomes
\begin{equation}\label{bin2}
P_{RR,X}(T)= \binom{\frac{T}{2 q_X(\lll,\mm)}+n+1}{n}.
\end{equation}
By  Lemma~\ref{star} below (applied with $c=n!$, $c'=1$, and $q=2 q_X(\lll,\mm)$), we get    $q_X(\lll,\mm)=1$ and the quadratic form~$q_X$ is even.\
The value of $c_X$ is then derived from~\eqref{pol} and the polynomial $P_{RR,X}(T)$ from~\eqref{bin2}.
\end{proof}

It remains to prove the arithmetical result   used at the end of the proof above.\ We prove a bit more  than what we actually used above but we will need this stronger statement for the proof of Theorem~\ref{th43}.

\begin{lemm}\label{star}
Let $X$ be a \hk\ manifold of dimension $2n$.\ Assume that there are positive integers $c$, $c'$, and $q$  such that $c'$ is divisible by no nontrivial $n$th powers and the polynomial $P(T)\coloneqq \frac{c}{c'}P_{RR,X}(qT)$ is  monic  with integral coefficients.\ Then, either $q=1$, or $q=2$  and the quadratic form $q_X$ is even.
\end{lemm}

\begin{proof}
Let $\alpha\in H^2(X,\Z)$ and write $\frac{ q_X(\alpha)}{q}=:\frac{r}{s}$, where $r$ and $s$ are relatively prime integers.\ One has
\begin{equation*}
    P_{RR,X}( q_X(\alpha))=\frac{c'}c\, P\Bigl(\frac{q_X(\alpha)}{q}\Bigr)
    =\frac{c'}c\, P\Bigl(\frac{r}{s}\Bigr)
\end{equation*}
and this is an integer by Lemma~\ref{lem11}.\ Since $P(T)$ is monic  with integral coefficients, $s^nP(\frac{r}{s})$ is an integer congruent to $r^n$ modulo $s$, hence prime to $s^n$.\ But $c's^nP(\frac{r}{s})= cs^nP_{RR,X}( q_X(\alpha))$ is divisible by $s^n$, hence so is $c'$.\ Our hypothesis implies $s=\pm1$, which proves that  $q$ divides all values that $q_X$ takes on $H^2(X,\Z)$.\
Since the  integral bilinear form associated with $q_X$ is  not  divisible, either $q=1$, or $q=2$  and the quadratic form $q_X$ is even.
\end{proof}

\subsection{The exceptional divisor}\label{subsec:ExceptionalDivisor}

We keep our \hk\ manifold $X$ of dimension $2n$ with a Lagrangian fibration $f\colon X\to \P^n$ and we set as above $L\coloneqq f^*\cO_{\P^n}(1)$, with first Chern class $\lll\in H^2(X,\Z)$.\
We assume further that there exists a line bundle $M$ on $X$ whose class~$\mm$ satisfies $\int_X \lll^{n}\mm^{n}=n!$ and $q_X(\lll,\mm)>0$.\

As in the proof of Theorem~\ref{th21}, $X$ is projective and we may assume  $q_X(\lll,\mm)=1$ and $q_X(\mm)=0$.\
Then $f_*M=\cO_{\P^n}(1)$ and $f_*(L^{-1}\otimes M )=\cO_{\P^n}$, hence the linear system  $| L^{-1}\otimes M |$ contains a single (effective) divisor~$E$  which still induces a principal polarization on the smooth fibers of $f$.\
It satisfies $q_X([E])=-2$ and $q_X([E],\mm)=-1$; in particular, $M$ is not nef.

\begin{lemm}\label{lem24}
Assume that $\NS(X)=\Z \lll \oplus \Z \mm$.
Then the divisor $E\in | L^{-1}\otimes M |$ is irreducible and reduced.
\end{lemm}

\begin{proof}
For any integer $k\ge 2$, the linear system $| L^{-k}\otimes M |$ is empty, because any divisor in that linear system would have negative $q_X$-intersection with $E$, and would therefore contain~$E$, but the linear system $|L^{-k+1}|$ is empty.\
It follows that $E$ has no vertical components.\ The fact that $E$ is   irreducible and reduced then follows from the relation $q_X([E],\lll)=1$.
\end{proof}

Since $q_X( [E])<0$, the prime divisor $E$ is therefore exceptional in the sense of \cite[d\'ef.~3.10]{bou} and
\cite[Definition~3.2]{mar} hence it spans an extremal ray in the effective cone $\Eff(X)$.\
Moreover, by~\cite[prop.~1.4 and rem.~4.3]{dru}, there is a birational isomorphism $\phi\colon X\isomdra X' $ (with~$X'$ smooth \hk) and a projective divisorial contraction $c\colon X'\to Y$ with exceptional divisor $\phi(E)$; moreover, the general fibers of $c\vert_{\phi(E)}$ are either smooth rational curves, or unions of two smooth rational curves meeting transversely at one point.\
The divisor $E$ is uniruled, its class $-\lll+\mm$ is primitive (because $q_X(\lll,[E])=1$), the reflection
\[
\alpha \longmapsto  \alpha+ q_X(\alpha,-\lll+\mm) (-\lll+\mm)
\]
is integral and a monodromy operator that permutes $\lll$ and $\mm$.\
Finally,  the class in $H_2(X',\Z)\isom H^2(X',\Z)^\vee\isom H^2(X,\Z)^\vee$ of a general (curve) fiber of $\phi(E)\to c(\phi(E))$ is given by the linear form $q_X(-\lll+\mm,\bullet)$; moreover, since $q_X(-\lll+\mm,\lll)=1$, this general fiber cannot be the union of two homologous curves, hence it is a smooth rational curve (\cite[Corollary~3.6]{mar}).\

\subsection{Examples}\label{subsec:SYZKnownExs}
 The next two examples show that Theorem~\ref{th21} applies to  \hk\ manifolds of K3$^{[n]}$ deformation type (Example~\ref{ex:K3n}) or of $\mathrm{OG10}$ deformation type (Example~\ref{ex:OG10}).

\begin{exam}\label{ex:K3n}
Let $S$ be a K3 surface with a primitive polarization $\hh_S$ of degree $\hh_S^2=2d$ and set $n=d+1$ (this is the genus of any curve in the linear system $ |\hh_S|$).\
Assume that the pair $(S,\hh_S)$ is very general, so that $\NS(S)=\Z \hh_S$.\
The smooth projective moduli space  $\M_0(S)\coloneqq M_{S}(0,\hh_S,0)$ parametrizes pairs consisting of a curve $C\in|\hh_S|$ and a torsion-free, rank-1 coherent sheaf on~$C$ of degree $n-1$, considered as a (torsion) sheaf on $X$.\
It is a \hk\ variety of K3$^{[n]}$ deformation type.\
There is a Lagrangian fibration
$
f\colon \M_0(S)\lra |\hh_S|=\P^{n}
$
that takes a sheaf to its support.\
The fiber of a smooth $C\in  |\hh_S|$ is the Jacobian $J^{n-1}(C)$.

The lattice $\NS(\M_0(S))$   is spanned by two isotropic vectors
$\lll\coloneqq c_1(f^*\cO_{\P^{n}}(1))$ and $  \mm$
 which satisfy $ q_X(\lll, \mm)=1$.\  Since the Fujiki constant is $ (2n-1)!!$, formula~\eqref{defia} gives $a = q_X(\lll, \mm)^n=1$.\  The fibers $J^{n-1}(C)$ have canonical theta divisors that fit together to define an effective divisor in $\M_0(S)$ which is the exceptional divisor $E$ from Section~\ref{subsec:ExceptionalDivisor}.
\end{exam}

\begin{exam}[R\'ios Ortiz]\label{ex:OG10}
According to~\cite{lsv}, there is a \hk\ manifold~$X$ of OG10 deformation type with a Lagrangian fibration \mbox{$f\colon X\to \P^5$} and an $f$-ample effective divisor~$\Theta$ on~$X$ that restricts to a principal polarization on the smooth fibers of $f$ (\cite[Proposition~5.3]{lsv}).\
Set $L\coloneqq f^*\cO_{\P^n}(1)$ and $M\coloneqq\cO_X(\Theta)$.\
Theorem~\ref{th21} shows that the Fujiki constant $c_X$ is equal to $9!!=945$ (it was originally computed in~\cite{rap}) and that
$
P_{RR,X}(T)= \binom{ \frac{T }{2  }+6}{5}$ (\cite{ort}).
\end{exam}

The next two examples deal with the other   known types of  \hk\ manifolds.

\begin{exam}\label{ex:Kumn}
 Let $A$ be an abelian surface with a polarization~$\hh_A$ of type $(1,d)$, with $ d\ge 3$.\
Assume that the pair $(A,\hh_A)$ is very general, so that $\NS(A)=\Z \hh_A$.\
The smooth projective moduli space  $\M_0(A)\coloneqq M_{A}(0,\hh_A,0)$ parametrizes pairs consisting of a curve $C\subset A$ with class~$\hh_A$ and a torsion-free, rank-1 coherent sheaf on~$C$ of degree $d$.\
The fibers of its (surjective) Albanese map $\M_0(A)\to A\times \widehat A$   are all isomorphic to the same \hk\ manifold $\K_0(A)$ of dimension $2n\coloneqq 2d-2$ which is of generalized Kummer deformation type (\cite[Theorem~0.2(1)]{yos:moduli}).\
There is a Lagrangian fibration
$
f\colon \K_0(A) \to \P^{n}
$
that takes a sheaf to its support.\
The fiber of a smooth $C\num \hh_A$ is the kernel of the Abel--Jacobi map $J^{d}(C)\to A$.\

By~\cite[Theorem 0.2(2)]{yos:moduli}, the lattice $\NS(\K_0(A))$ is spanned by the two isotropic vectors $\lll\coloneqq c_1(f^*\cO_{\P^{n}}(1))$ and $  \mm=c_1(M)$ which satisfy $q_X(\lll, \mm)=1$.\
Since the Fujiki constant is $(n+1) (2n-1)!!$, formula~\eqref{defia} gives $a=n+1$.
\end{exam}

\begin{exam}
Examples of \hk\ manifolds of  OG6 deformation type with a Lagrangian fibration are described in \cite{rap}.
\end{exam}


\section{Conjecture~\ref{conjnum} for hyper-K\"ahler fourfolds}\label{sec:IsotropicClasses}

The main result of  this section is the proof of  Conjecture~\ref{conjnum} in dimension $4$ (see Theorem~\ref{th43a}).\ It is a simple consequence of the work \cite{gua} of Guan  who gave a list of possible Betti numbers for hyper-K\"ahler fourfolds.

Let $X$ be a \hk\ fourfold.\ Following \cite[Section~2.4]{jia}, we
 set
\begin{equation}\label{ax}
A_X\coloneqq\int_X\td^{1/2}(X)= \frac 1 {5760} ( 7c_2^2(X) -4c_4(X) ) = \frac 18 \Bigl(7-\frac 1 {432} c_4(X)\Bigr).
\end{equation}
It is known that $c_4(X)$ is divisible by 12 (see, for example, \cite[Proposition 2.4]{gr}), hence $288 A_X$ is an integer.

\begin{lemm}\label{lemmguan}
Let $X$ be a \hk\ fourfold.\ Then,
\begin{enumerate}[{\rm (a)}]
\item\label{aax} either $b_2(X) =23$,  $b_3(X) =0$, and the Hodge numbers of $X$ are those of the Hilbert square of a K3 surface, in which case $ c_4(X)= 324$ and $A_X=\frac{25}{32}$,
\item\label{bax} or $b_2(X) \le 8$, in which case $c_4(X)\le 144$ and  $\frac{5}{6}\le A_X\le \frac{131}{144}$.
 \end{enumerate}
In particular, if $t\in [0,\frac13)$, then $4A_X-t $ is an integer only when $t=\frac18$ and $A_X=\frac{25}{32}$.
\end{lemm}

\begin{proof}
This follows from~\eqref{ax}, the relation $c_4(X)=3(4b_2(X)+16-b_3(X))$, and the list of possible Betti numbers given in  \cite[Main Theorem]{gua}.
\end{proof}

We now introduce classes $\lll,\mm\in H^2(X,\Z)$ such that $q_X(\lll)=0$ and define $a=\frac12\int_X\lll^2\mm^2$ as in~\eqref{defa2}.\ By Lemma~\ref{lemmaa} and~\eqref{defia}, it is a  nonnegative integer.

\begin{lemm}\label{lemmax}
The number $\sqrt{2aA_X} $ is   rational.
\end{lemm}

\begin{proof}
By~\eqref{defia}, we have $c_Xq_X(\lll,\mm)^2=3a$.\ Using Section~\ref{subsec:RRpolynomial} and~\cite[Lemma~5.7]{jia}, we obtain
\begin{equation}\label{prr}
   P_{RR,X}(T)=\frac{c_X}{24}T^2+ T\sqrt{\frac23c_XA_X} +3=\frac{a}{8}\Bigl(\frac{T}{q_X(\lll,\mm)}\Bigr)^2+ \sqrt{2aA_X}\Bigl(\frac{T}{q_X(\lll,\mm)}\Bigr) +3.
\end{equation}
This implies that $\sqrt{2aA_X} $ is a rational number.
\end{proof}

We are now ready to prove (a stronger form of) Conjecture~\ref{conjnum} (which corresponds to the case $a=1$) in dimension 4.\ Analogous, but weaker,  results for low values $a$ will be given in Theorem~\ref{th43}.

\begin{theo}\label{th43a}
Let $X$ be a  \hk\ fourfold with classes~$\lll,\mm\in H^2(X,\Z)$ such that $\int_X \lll^{4}=0$
and   $\int_X \lll^{2}\mm^{2}=2$.\
 The Chern and Hodge numbers of $X$ are those of the Hilbert square of a K3 surface  and all the  conclusions of Theorem~\ref{th21} hold.
\end{theo}

\begin{proof}
Changing   $\mm$ into $\pm\mm+r\lll$, for some $r\in\Z$, if necessary, we may assume $-q_X(\lll,\mm)< q_X(\mm)\le q_X(\lll,\mm)$ or, equivalently, $\gamma\coloneqq \frac{q_X(\mm)}{q_X(\lll,\mm)}\in (-1,1] $.\

As in~\eqref{defP}, we introduce the polynomial
$$P(k)\coloneqq P_{RR,X}(q_X(k\lll+ \mm))=P_{RR,X}(2kq_X(\lll,\mm)+q_X(\mm)).$$
Using~\eqref{prr}, we compute
\begin{align*}
 P(k)&= \frac{1}{8}(2k+\gamma)^2+ \sqrt{2A_X}(2k+\gamma) +3  \\
 &=\frac{1}{2}k^2+\Bigl( \frac{1}{2}\gamma+2\sqrt{2A_X} \Bigr)k+\frac{1}{8}\gamma^2+ \gamma\sqrt{2A_X}+3\\
 &=:\frac{1}{2}k^2 +  b  k +c.
\end{align*}
Since $P$ takes integral values on integers, $\frac{1}2+b=P(1)-P(0)$ and $c=P(0)$ are integers; we write $b=\frac12+b'$, with $b'\in\Z$.

We also note that
\begin{align*}
    4A_X-\frac{b^2}{2}&= 4A_X-\frac{1}{2}\Bigl( \frac{1}{2}\gamma+2\sqrt{2A_X} \Bigr)^2\nonumber\\
    &=4A_X-\frac{1}{2}\Bigl( \frac{1}{4}\gamma^2+2\gamma\sqrt{2A_X}+8A_X \Bigr)\\
    &=3-c\nonumber
\end{align*}
is an integer, hence so is $4A_X-\frac{1}{8}$.\ By Lemma~\ref{lemmguan}, this is only possible when $A_X=\frac{25}{32}$ and the Chern and Hodge numbers of $X$ are those of the Hilbert square of a K3 surface.\ We also have
$$\frac12+b'=b=\frac{1}{2}\gamma+\frac52,$$
which implies that $\gamma$ is an even integer.\ Since $\gamma \in (-1,1] $, we obtain $\gamma=0$,  hence $q_X(\mm)=0$, $b=\frac52$, and $c=3$.\ By  the very definition~\eqref{defP} of $P$, one has
\[
P_{RR,X}(2kq_X(\lll,\mm))=P(k)= \frac{1}{2}(k^2 +  5  k +6)=\binom{k+3}{2}.
\]
By Lemma~\ref{star} (applied with $c=2$, $c'=1$, and $q=2q_X(\lll,\mm)$), we get $q_X(\lll,\mm)=1$, the quadratic form $q_X$ is even, and all the conclusions of Theorem~\ref{th21} hold.
\end{proof}


\section{On the SYZ conjecture in dimension 4}\label{sec:SYZ}

In this section we state our two main results, Proposition~\ref{prop:CaseC1} and Proposition~\ref{prop:CaseC2}, on the SYZ conjecture for very general \hk\ fourfolds when numerically we expect that there exists a principal polarization.\
We will prove them respectively in Section~\ref{sec:TwoNefIsotropic} and Section~\ref{sec:DivisorialContraction}; we start this section by reviewing a few general results on semi-ample line bundles and Lagrangian fibrations in dimension~4.

\subsection{Results of Kawamata, Fujino, Matsushita, Fukuda, Huybrechts and Xu}\label{subsec:Lagrangian4folds}

Let~$X$ be a \hk\ manifold of dimension~$2n$.\
As we noted in Section~\ref{subsec:BBF}, given a nontrivial nef line bundle $L$ with primitive class $\lll$ such that $\int_X \lll^{2n}=0$, we have $\lll^{n}\ne0$ and $\lll^{n+1}=0$, hence the {\em numerical dimension} $\nu(X,L)$ is $n$.\

The {\em Iitaka dimension} $\kappa(X,L)$, that is, the dimension of the image of the  rational map $\phi_{L^{\otp k}}\colon X\dra \P(H^0(X,L^{\otp k})^\vee)$ for $k$ sufficiently large and divisible, satisfies $\kappa(X,L)\le \nu(X,L)$.\
If there is equality (that is, in our case, if $\kappa(X,L)=n$), the line bundle $L$ is {\em good} in the sense of \cite[Section~1]{kaw} (the current terminology is {\em abundant}; see \cite[Definition~2.2]{fuj}).\
In that case, a theorem of Kawamata (\cite[Theorem~6.1]{kaw} in the algebraic case and \cite[Theorem~4.8]{fuj} for a simpler proof also valid in the analytic case) says that $L$ is semi-ample: for $k$ sufficiently large and divisible, the sections of~$L^{\otp k}$  define a {\em morphism} $f\colon X\to B$ with connected fibers which, by \cite[Theorem~1]{mat}, is a Lagrangian fibration.

When the dimension of $X$ is 4, there are further results:
\begin{itemize}
\item assuming only $h^0(X,L^{\otp k})\ge 2$ for some $k>0$ (that is,  $\kappa(X,L)>0$), the line bundle $L$ is semi-ample (\cite[Theorem~1.5]{fuk});
\item the base $B$ of the Lagrangian fibration is isomorphic to $\P^2$ (\cite{hx}).
\end{itemize}
Assume now that $L$ is semi-ample.\ Since $\lll$ is primitive, one can   write $f^*\cO_{\P^2}(1)=L^{k_L}$ for some positive integer $k_L$.\
By~\cite[Theorem~1.3]{mathigher} and the projection formula, we have
$R^qf_*L^{k_L} \isom \Omega_{\P^2}^q(1)$ for all $q$.\
Since we know that $h^p(\P^2,\Omega_{\P^2}^q(1))=0$ except for $p=q=0$, the Leray spectral sequence gives
\begin{equation}\label{kL}
h^0(X,L^{k_L}) = h^0(\P^2, \cO_{\P^2}(1))=3\quad\textnormal{and}\quad h^i(X,L^{k_L}) =0\ \ \textnormal{for all } i>0.
\end{equation}
In particular, the map $f $ is the map $ \phi_{L^{k_L}}$.\ We will need the following elementary observation.

\begin{lemm}\label{lem:H0large}
Let $L$ be a nef line bundle on a  \hk\ fourfold $X$ with $\kappa(X,L)>0$.\
We assume that its class $\lll\in\NS(X)$ is primitive and satisfies $\int_X\lll^4=0$.\
Then,
\begin{enumerate}[{\rm (a)}]
\item either $h^0(X,L)\le1$;
\item or $k_L=1$, $h^0(X,L)=3$, and~$L$ is globally generated.
\end{enumerate}
\end{lemm}

\begin{proof}
Let us assume  $k_L>1$ and $h^0(X,L)\ge2$.\ Let $\sigma,\tau\in H^0(X,L)$ be
linearly independent sections.\ The $k_L+1$ sections  $\sigma^{k_L},\sigma^{k_L-1}\tau,\dots,\tau^{k_L}$ are then linearly independent in $H^0(X,L^{\otp k_L})$.\ Since this space has dimension 3 (see~\eqref{kL}), we have $k_L=2$ and the image of the map $f=\phi_{L^{k_L}}\colon X\to \P^2$ is a conic.\ This  contradicts the surjectivity of $f$.
\end{proof}

\subsection{Cones of divisors}\label{subsec:ConeDivisors}
Let $X$ be a \hk\ manifold of dimension $2n$ with   classes~$\lll, \mm\in \NS(X)$ such that $q_X(\lll)=q_X(\mm)=0$ and $q_X(\lll,\mm)>0$.\
We assume moreover that $\NS(X)=\Z \lll\oplus \Z \mm$.

The (closed) \emph{positive cone} $\Pos(X)\subset\NS(X)\otimes\R$ is defined as the closure of the set of   classes of divisors with positive self-intersection and   positive intersection with a K\"ahler class for the form  $q_X$.\
Under our assumptions, after possibly changing signs, the positive cone is then
\[
\Pos(X) =\R_{\ge 0}\lll + \R_{\ge 0} \mm.
\]

The (closed) \emph{movable cone} $\Mov(X)\subset\NS(X)\otimes\R$ is defined as the closure of the cone generated by classes of effective divisors whose  base locus has codimension $\geq2$.\
We have an inclusion $\Mov(X)\subset\Pos(X)$.\
To determine the movable cone, we need to understand prime exceptional divisors, namely reduced and irreducible divisors with negative self-intersection for the quadratic form $q_X$ (see~\cite[Lemma~6.22]{mar:survey}).

\begin{lemm}\label{lem:PrimeExceptional}
Let $E$ be a prime exceptional divisor on $X$.\
We have $[E]=\pm(-\lll+\mm)$ in $\NS(X)$.
\end{lemm}

\begin{proof}
Let us write $[E]=t\lll+u\mm$, with $t,u\in\Z$.
By~\cite[Corollary~3.6]{mar}, the class
\[
-2\frac{q_X([E],\bullet)}{q_X([E])} = -\frac{q_X(t\lll+u\mm,\bullet)}{tuq_X(\lll,\mm)}
\]
is 
in $H_2(X,\Z)$.\
By applying it to $\lll$ and $\mm$, we get $|t|=|u|=1$.\ Since $q_X([E])<0$, we have $t=-u$, as we wanted.
\end{proof}

By Lemma~\ref{lem:PrimeExceptional}, after possibly permuting $\lll$ and $\mm$, we can assume that the ray $\R_{\ge0}\lll$ is extremal for the movable cone $\Mov(X)$.\
By~\cite[Theorem~7]{hats}, upon replacing $X$ with a birational model if necessary, which by~\cite[Theorem 4.6]{HuyHK} does not change the deformation type, we can assume that this ray is also extremal for the nef cone $\Nef(X)$.\
Therefore, one can write
\begin{align*}
\Nef(X)&=\R_{\ge 0}\lll + \R_{\ge 0} (t_{\textnormal{nef}}\lll+\mm),\\
\Mov(X)&=\R_{\ge 0}\lll + \R_{\ge 0} (t_{\textnormal{mov}}\lll+\mm),
\end{align*}
where $t_{\textnormal{nef}}\ge t_{\textnormal{mov}}\ge 0$ are rational numbers.
By~\cite[prop.~4.4]{bou}, the movable and pseudo-effective cones are mutually $q_X$-dual, hence
\[
\Psef(X)=\R_{\ge 0}\lll + \R_{\ge 0} (-t_{\textnormal{mov}}\lll+\mm).
\]


By Lemma~\ref{lem:PrimeExceptional}, we deduce  the following.

\begin{lemm}\label{lem61}
Either $t_{\textnormal{mov}}=0$, or $t_{\textnormal{mov}}=1$ and there is a prime effective divisor $E$ in $X$ with class $-\lll+\mm$.
\end{lemm}

Note that, if $t_{\textnormal{mov}}=1$, the nef cone coincides with the movable cone if and only if the class $\lll+\mm$ is nef; this class is then semi-ample by Kawamata's Base-Point-Free Theorem and some multiple of it defines a divisorial contraction $c\colon X\to Y$ with exceptional divisor $E$.

\subsection{The two key propositions}\label{subsec:HKNefCone}

In the rest of this article, we will be mostly concerned with triples $(X,\lll,\mm)$ satisfying the following set of properties:
\begin{equation}\label{hypo}
    \left\{
    \begin{array}{l}
    X \hbox{ is a  \hk\ fourfold;}\\
   \lll, \mm\in \NS(X);\\
   \int_X \lll^{4}=0 \hbox { and }\int_X \lll^{2}\mm^{2}=2.
    \end{array}
    \right.
\end{equation}
We will always denote by $L$ and $M$ the line bundles on $X$ representing $\lll$ and $\mm$.

By Theorem~\ref{th43a}, properties~\eqref{hypo} imply the following other set of properties (after possibly changing $\mm$ into $-\mm$ and adding to it a multiple of $\lll$):
\begin{equation}\tag{\theequation$^\prime$}\label{hypoprime}
\begin{small}    \left\{
    \begin{array}{l}
   q_X(\lll)=q_X(\mm)=0\hbox{ and } q_X(\lll,\mm)=1; \\
   \hbox{the Chern and Hodge numbers of $X$ are those of the Hilbert square of a K3 surface;}\\
    \hbox{the canonical map $\Sym^2\!H^2(X,\Q)\isomto H^4(X,\Q)$ is an isomorphism;}\\
    P_{RR,X}(T)= \binom{ \frac{T}{2}+3}{2};\\
    \hbox{the quadratic form $q_X$ is even};\\
c_X =3, \hbox{ so that $\int_X\alpha^4=3q_X(\alpha)^2$ for all } \alpha\in H^2(X,\Z).
    \end{array}
    \right.
\end{small}
\end{equation}
So when we assume~\eqref{hypo}, we will always assume that~\eqref{hypoprime} also holds.

Following \cite[Section~3]{og1}, we define the dual $q_X^\vee \in H^4(X,\Q)$ of the quadratic form $q_X$.\
Since $b_2(X)=23$, it satisfies (\cite[Proposition~2.2]{og1})
\begin{equation}\label{qxc}
\int_X q_X^\vee\cdot q_X^\vee=575\quad,\qquad \forall \alpha,\beta\in H^2(X,\Z)\quad \int_X q_X^\vee\cdot\alpha\cdot\beta=25 q_X(\alpha,\beta).
\end{equation}
Finally, we have (\cite[(3.0.45)]{og1})
\begin{equation}\label{c2}
c_2(X)=\frac 65 q_X^\vee.
\end{equation}

For any  cohomology class
$\eta\in H^4(X,\Q)$, we define a symmetric intersection matrix
\begin{eqnarray}\label{eqmatriceO}
M_\eta\coloneqq
\begin{pmatrix}
\int_X\eta\lll^2& \int_X\eta\lll\mm \\
\int_X\eta\lll\mm & \int_X\eta\mm^2
\end{pmatrix}
\end{eqnarray}
with rational coefficients (which are integers if $\eta$ is integral).\
We have for example
\begin{eqnarray}\label{eqdematl2m2}
M_{\lll^2}=\begin{pmatrix}
0 & 0 \\
0 & 2
\end{pmatrix},\quad M_{\lll\mm}=\begin{pmatrix}
0 & 2 \\
2 & 0
\end{pmatrix}, \quad M_{\mm^2}=\begin{pmatrix}
2 & 0 \\
0 & 0
\end{pmatrix}, \quad M_{q_X^\vee}=\begin{pmatrix}
0 & 25 \\
25 & 0
\end{pmatrix}.
\end{eqnarray}

The manifolds $X$ satisfying~\eqref{hypo} are all projective and, by the surjectivity of the period map, they form an irreducible 18-dimensional family.\
For very general members of this family, we show that the groups of Hodge classes are very simple.

\begin{prop}\label{prop:SurfacesOnHK4}
A very general triple $(X,\lll,\mm)$ satisfying~\eqref{hypo} has the following properties:
\begin{enumerate}[{\rm (a)}]
\item\label{aaa} $\NS(X)=\Z\lll\op \Z\mm$;
\item\label{bbb} the group of degree-2 rational Hodge classes is
$$\Hdg^2(X,\Q)={\Sym}^2{\NS}(X)_{\Q}\op\Q q_X^\vee=\Q\lll^2\oplus \Q\lll\mm\op\Q\mm^2\op\Q q_X^\vee.$$
\end{enumerate}
\end{prop}

\begin{proof}
 Item~\ref{aaa} is classical.\ As for   item~\ref{bbb}, the isomorphism $\Sym^2H^2(X,\Q)\isom H^4(X,\Q)$ from~\eqref{hypoprime} induces a decomposition
\begin{eqnarray}\label{eqdecompH4}
H^4(X,\Q)={\Sym}^2H^2(X,\Q)_{\rm tr}\oplus (H^2(X,\Q)_{\rm tr}\otimes \NS(X)_{\Q})\oplus {\Sym}^2{\NS}(X)_{\Q},
\end{eqnarray}
where $H^2(X,\Q)_{\rm tr}\coloneqq \NS(X)_{\Q}^\perp$.\ Since  $X$ is very general, the Mumford--Tate group of the Hodge structure on $H^2(X,\Q)_{\rm tr}$ is the orthogonal group of the  form $q_X$, so that the only Hodge classes in ${\Sym}^2H^2(X,\Q)_{\rm tr}\subset H^4(X,\Q)$ are multiples of the class $q_X^\vee$.\ It follows from the decomposition~\eqref{eqdecompH4} that the space of rational degree-2 Hodge classes on $X$ is generated by ${\Sym}^2{\NS}(X)_{\Q}$ and $q_X^\vee$.
\end{proof}

We described in   Section~\ref{subsec:ConeDivisors} the general structure of the cones of divisors.\ In our case, we use an idea similar to the one used in the proof of   Proposition~\ref{prop:SurfacesOnHK4} to make this description more precise.

\begin{prop}\label{lem:NefCone4folds}
Let $(X,\lll,\mm)$ be a triple satisfying~\eqref{hypo} and such that~$\NS(X)=\Z\lll\oplus\Z\mm$.\
The nef and movable cones of $X$ coincide.\
Equivalently, $X$ has no nontrivial \hk\ birational models.
\end{prop}

\begin{proof}
We follow the proof of \cite[Theorem~22]{hats}, whose argument we briefly review.\
We assume for a contradiction that the nef and movable cones of $X$ are different.\
By \cite[Theorem 1.1]{wiwi}, there exists a Mukai flop on $X$, hence a Lagrangian plane $\P^2\subset X$.\
Let $\ell\in H_2(X,\Z)$ be the class of a line in $ \PP^2$.

Since the lattice $(\NS(X),q_X)$ is unimodular (by~\eqref{hypoprime}, it is a hyperbolic plane), there exists \mbox{$A\in\NS(X)$} such that
\begin{equation}\label{eq:DualLine}
\forall B\in  \NS(X)\qquad q_X(A,B) =  B\cdot \ell.
\end{equation}
Since the class $\ell$ is of type $(1,1)$, it is orthogonal to $H^{2,0}(X)$, hence to the transcendental lattice  as defined in \cite[Definition~3.2.5]{huyk3}.\
But the transcendental lattice is also the orthogonal complement of $\NS(X)$, hence~\eqref{eq:DualLine} remains valid for all $B$ in $\NS(X)^\perp_\Q$, hence for all $B$ in~$H^2(X,\Q)$.

As explained in~\cite[(8)]{hats} and~\cite[Section~3]{og1}, upon replacing $X$ with a very general deformation for which the class $\ell$ remains a Hodge class (hence for which the plane $\P^2$ deforms along; see~\cite{voi}), we can write
\[
[\P^2] = t A^2 + u q_X^\vee \in H^4(X,\Q),
\]
for some $t,u\in\Q$.\
We now compute the numbers
\begin{equation*}
[\P^2]\cdot[\P^2],\qquad c_2(X)\cdot [\P^2],\qquad A^2\cdot [\P^2]
\end{equation*}
in order to evaluate $t$, $u$, and $q_X(A)$.\

Since $\P^2\subset X$ is Lagrangian, we have $N_{\P^2/X}\cong\Omega^1_{\P^2}$, and so (here we use $c_X=3$ from~\eqref{hypoprime}, and~\eqref{qxc})
\begin{equation}\label{eq:NefCone4folds1}
3 = c_2(N_{\P^2/X}) = [\P^2]\cdot [\P^2] = 3 t^2 q_X(A)^2 + 50 t u q_X(A) + 575 u^2.
\end{equation}
From the normal bundle exact sequence for $\P^2\subset X$, we get (here we use~\eqref{c2} and~\eqref{qxc})
\begin{equation}\label{eq:NefCone4folds2}
-3 = c_2(X)\cdot [\P^2] = \frac 65 \left( 25 t q_X(A) + 575 u \right)= 30 \left(t q_X(A) +  25 u \right).
\end{equation}
Finally, since $A\vert_{\P^2} = (A\cdot \ell)\ell$, we get from~\eqref{eq:DualLine} the relations
\begin{equation}\label{eq:NefCone4folds3}
q_X(A)^2  =  (A\cdot \ell)^2=  (A\vert_{\P^2})^2 = A^2\cdot [\P^2] =  3 t q_X(A)^2 + 25 u q_X(A).
\end{equation}

Solving~\eqref{eq:NefCone4folds1}, \eqref{eq:NefCone4folds2}, and \eqref{eq:NefCone4folds3} for $t$ and $u$, we get the equation
\[
92 q_X(A)^2 + 20  q_X(A) - 525 = 0,
\]
which has no integral solutions.\
This a contradiction, which proves the proposition.
\end{proof}

Under the assumptions of Proposition~\ref{lem:NefCone4folds},  permuting $\lll$ and $\mm$ if necessary, we can assume that~$\lll$ is nef  and write, as in Lemma~\ref{lem61} (by Proposition~\ref{lem:NefCone4folds}, we do not need to replace~$X$ by a birational model  as in Section~\ref{subsec:ConeDivisors})
\begin{align}
\Pos(X) &=\R_{\ge 0}\lll + \R_{\ge 0} \mm,\nonumber\\
\Mov(X)=\Nef(X)&=\R_{\ge 0}\lll + \R_{\ge 0} (t_0\lll+\mm),\label{cones}\\
\Psef(X)&=\R_{\ge 0}\lll + \R_{\ge 0} (-t_0\lll+\mm),\nonumber
\end{align}
where $t_0\in\{0,1\}$.\
In sum, there are two cases for a  triple $(X,\lll,\mm)$ satisfying~\eqref{hypo} and such that~$\NS(X)=\Z\lll\oplus\Z\mm$:
\begin{enumerate}[label={\rm(C\arabic*)},ref=(C\arabic*)]
\item\label{caseC1} either $t_0=0$, the class $\mm$ is nef and all cones of divisors are equal;
\item\label{caseC2} or $t_0=1$ and Lemma~\ref{lem61} says that there exists a unique prime divisor~$E$ whose class is not in the positive cone.\ We have $E\in |L^{-1}\otimes M|$ and the sections of \mbox{$(L\otimes M)^{\otp k}$} define, for $k\gg 0$, the divisorial contraction $c\colon X\to Y$ of $E$.\ The discussion of Section~\ref{subsec:ExceptionalDivisor} still applies: a general fiber of $E\to c(E)$ is a smooth rational curve and the reflection about the hyperplane $(-\lll+\mm)^\bot$ is a monodromy operator that permutes~$\lll$ and~$\mm$.
\end{enumerate}

In short, the class $\mm$ is nef in case~\ref{caseC1} and not nef in case~\ref{caseC2}.\
In the former case, we have the following result.

\begin{prop}\label{prop:CaseC1}
Let $(X,\lll,\mm)$ be a very general triple satisfying~\eqref{hypo}
and assume in addition that we are in case~{\rm\ref{caseC1}}.\
 Then one of the following  statements holds:
\begin{enumerate}[{\rm (a)}]
       \item  $H^0(X,L)\ne 0$, $H^0(X,M)\ne 0$, and  there is a Lagrangian fibration \mbox{$f\colon X\to\P^2$} with $f^*{\cO_{\P^2}}(1)\isom L$ or $M$;\label{prop64a}  
    \item  at least one of $H^0(X,L)$ or $H^0(X,M)$ is trivial and  the image of the rational map $\phi_{L\otimes M}\colon X\dra \P^5$ is rationally connected.\label{prop64b}
\end{enumerate}
\end{prop}

We will prove Proposition~\ref{prop:CaseC1} in Section~\ref{proof}.\
As a consequence of the results in Section~\ref{subsec:ExceptionalDivisor} and~\cite{voisinfib}, we obtain that case~\ref{caseC1} actually does not occur.

\begin{coro}\label{cor:NoTwoIsotropicNef}
Let $(X,\lll,\mm)$ be a very general triple satisfying~\eqref{hypo}.\
Case~{\rm\ref{caseC1}} does not occur.
\end{coro}

\begin{proof}
By~\cite[Theorem 4.2]{voisinfib}, in case~\ref{caseC1}, if $H^0(X,L)$ or $H^0(X,M)$ is $0$, the image of the rational map $\phi_{L\otimes M}\colon X\dra \P^5$ cannot be rationally connected.\ 
Hence we are in case~\ref{prop64a} of  Proposition~\ref{prop:CaseC1}: there exists a Lagrangian fibration \mbox{$f\colon X\to\P^2$} with $f^*{\cO_{\P^2}}(1)\isom L$ or~$M$.\
The discussion of Section~\ref{subsec:ExceptionalDivisor} then applies, showing that $M$ and $L$ cannot be both nef,    contradicting~\ref{caseC1}.
\end{proof}

In case~\ref{caseC2}, we have the following result.

\begin{prop}\label{prop:CaseC2}
Let $(X,\lll,\mm)$ be a very general triple
satisfying~\eqref{hypo} and assume in addition that we are in case~{\rm\ref{caseC2}}.\
 There exist a positive integer~$k_L$ and a Lagrangian fibration \mbox{$f\colon X\to\P^2$} such that $f^*{\cO_{\P^2}}(1)\isom L^{\otp {k_L}}$.
\end{prop}

We will prove Proposition~\ref{prop:CaseC2} in Section~\ref{sec:DivisorialContraction}.\
We will then show in Section~\ref{sec:OGrady} that   $k_L=1$ and $X$ is of  K3$^{[2]}$ deformation type, thus completing the proof of both Theorem~\ref{thm:MainThm} and Theorem~\ref{thm:SYZ}.

We note that Corollary~\ref{cor:NoTwoIsotropicNef} and Proposition~\ref{prop:CaseC2} do already imply a slightly weaker version of Theorem~\ref{thm:SYZ}, where we have no control on $k_L$.\
We include the proof since it does not use the results of Section~\ref{sec:OGrady} and  might apply to more general situations.

\begin{coro}[\HK\ SYZ conjecture]\label{cor:SYZ}
Let $X$ be a  \hk\ fourfold and let~$L$ be a nef line bundle on $X$.\ Set $\lll\coloneqq c_1(L)$.\ Assume   $\int_X \lll^4=0$ and that there exists a class $\mm\in H^2(X,{\Z})$ with $\int_X\lll^2\mm^2=2$.\
There exists a Lagrangian fibration $f\colon X\to\P^2$ with $f^*{\cO_{\P^2}}(1)\isom L^{\otp k_L}$ for some positive integer $k_L$.
\end{coro}

\begin{proof}
We can argue as in the beginning of the proof of Theorem~\ref{th21} and consider the moduli space ${\mathfrak M}_\lll$ of marked hyper-K\"ahler manifolds with a fixed $(1,1)$-class $\lll$.\
Then there exist  points $0,0^\prime\in {\mathfrak M}_\lll$ such that $(\cX_0,\cL_0)=(X,L)$ and $(X^\prime,L^\prime)\coloneqq (\cX_{0^\prime},\cL_{0^\prime})$ has the property that its N\'eron--Severi group is generated by the classes $\lll$ and $\mm$, $L^\prime$ is nef, and $X^\prime$ is very general in the sense of Proposition~\ref{lem:NefCone4folds}.

By Proposition~\ref{prop:CaseC2}, some power $L^{\prime\,k}$ defines a Lagrangian fibration $X^\prime\to\P^2$.\
According to~\cite[Lemma 2.4]{matlag}, this implies that the line bundle $\cL_t$ is semi-ample for all fibers~$\cX_t$ of Picard rank one.\
In particular, for   very general points $t\in{\mathfrak M}_\lll$, one has $h^0(\cX_t,\cL_t^{k_t})\geq2$ for some positive integer $k_t$.\
Hence, the  countable union  of the  closed sets $\{t\in{\mathfrak M}_\lll\mid h^0(\cX_t,\cL_t^{k})\geq2\}$, for all $k\in\Z_{>0}$,  contains all very general points.\
This is enough to conclude that there exists one $k$ for which the corresponding set is all of ${\mathfrak M}_\lll$ and, in particular $h^0(X,L^k)\geq2$.\
Hence, $L$ is semi-ample by Section~\ref{subsec:Lagrangian4folds}.
\end{proof}


\section{The case of two nef isotropic classes}\label{sec:TwoNefIsotropic}

Let $(X,\lll,\mm)$ be a very general triple satisfying~\eqref{hypo}.\
By Proposition~\ref{prop:SurfacesOnHK4}, we have  $\NS(X)=\Z\lll\oplus\Z\mm$.\
We assume in this section that we are in case~\ref{caseC1}: the class $\mm$ is isotropic and both~$\lll$ and $\mm$ are nef.\ Our aim is to prove Proposition~\ref{prop:CaseC1} (in Section~\ref{proof}).

As observed in Section~\ref{subsec:HKNefCone}, all cones of divisors are equal and the class $p \lll+q \mm$ is ample on $X$ for all integers $p,q>0$.\
By Kodaira vanishing, we get
\begin{equation}\label{eq:KodairaVanishingTwoNef}
h^0(X,L^{\otp p}\otimes M^{\otp q})=\chi(X,  L^{\otp p}\otimes M^{\otp q})=P_{RR,X}(2pq) =\binom{pq+3}{2}
\end{equation}
and in particular
\[
h^0(X,L\otimes M)=6,\quad h^0(X, L^{\otp 2}\otimes M)=10,\quad h^0(X, L^{\otp 3}\otimes  M^{\otp 2})=36.
\]
The next two sections and the paper  \cite{voisinfib} are devoted to the  study of  the induced rational map
\[
\phi_{L\otimes M}\colon X\dra \P^5.
\]
This map was studied by O'Grady in~\cite[Section~4]{og1} when~$X$ is   of~K3$^{[2]}$ numerical type, and for a very general deformation of the pair $(X,L\otimes M)$.\

\subsection{Case where $L$ and $M$ both have nonzero sections}\label{nonz}

The following lemma will allow  us to apply Lemma~\ref{lem:H0large}.\ We keep the same hypotheses: $(X,\lll,\mm)$ is a very   general triple satisfying~\eqref{hypo} and we are in case~\ref{caseC1}.

\begin{lemm}\label{lecasholhoM}
Assume that $H^0(X,L)\ne 0$ and $H^0(X,M)\ne 0$.\
Then
\begin{equation*}
h^0(X,L)+h^0(X,M)\geq 3,
\end{equation*}
hence either $h^0(X,L)\geq 2$ or $h^0(X,M)\geq2$.
\end{lemm}

\begin{proof}
Let $\sigma$ be a nonzero section of $L$ and let $\tau$ be a nonzero section of $M$.\
The divisors $D_\sigma$ of $\sigma$ and $D_\tau$ of $\tau$   have no common component, since the class of any effective divisor on $X$ is an integral combination of $\lll$ and $\mm$ with nonnegative coefficients.\
It follows that
\[
\Sigma_{\sigma\tau}\coloneqq D_\sigma\cap D_\tau
\]
is a surface whose ideal sheaf $\cI_{\Sigma_{\sigma\tau}}$ admits the Koszul resolution
\begin{equation*}\label{eqkosresSigma}
0\to  (L\otimes M)^{-1}\to    L^{-1}\oplus  M^{-1}\to   \cI_{\Sigma_{\sigma\tau}} \to  0.
\end{equation*}
If we tensor it by $L\otimes M$, the associated  exact sequence in cohomology gives
\begin{equation}\label{eqineq1}
h^0(X,\cI_{\Sigma_{\sigma\tau}}(L\otimes M))=h^0(X,L)+h^0(X,M)-1.
\end{equation}
If we tensor the resolution by $L^{\otp 2}\otimes M^{\otp 2}$, we get, using~\eqref{eq:KodairaVanishingTwoNef},
\begin{equation}\label{eqineq2}
\begin{split}
h^0(X,\cI_{\Sigma_{\sigma\tau}}(L^{\otp 2}\otimes M^{\otp 2}))&=h^0(X,L^{\otp 2}\otimes M)+h^0(X,L\otimes M^{\otp 2})-h^0(X,L\otimes M )=14,\\ h^1(X,\cI_{\Sigma_{\sigma\tau}}(L^{\otp 2}\otimes M^{\otp 2}))&=0.
\end{split}
\end{equation}
Using again~\eqref{eq:KodairaVanishingTwoNef}, we get $h^0(X,L^{\otp 2}\otimes M^{\otp 2})=21$, and, from~\eqref{eqineq2}, we deduce
\begin{equation}\label{eqineq3}
h^0(\Sigma_{\sigma\tau},(L^{\otp 2}\otimes M^{\otp 2})\vert_{\Sigma_{\sigma\tau}})=7.
\end{equation}

Assume by contradiction $h^0(X,L)+h^0(X,M)=2$.\
Then, by~\eqref{eqineq1}, we get
\begin{equation}\label{eqineq4}
\rk \bigl(H^0(X, L\otimes M)\lra  H^0(\Sigma_{\sigma\tau}, ( L\otimes M)\vert_{\Sigma_{\sigma\tau}})\bigr)= 5.
\end{equation}
In particular, the surface $\Sigma_{\sigma\tau}$ is not contained in the base locus of $ |L\otimes M|$.\

We prove now the following properties.

\begin{claim}\label{claimavecc2}
In the notation used above,
\begin{enumerate}[{\rm (a)}]
\item\label{enum:Sigma1} the surface $\Sigma_{\sigma\tau}$ is irreducible and reduced;
\item\label{enum:Sigma2} the image of the rational map
\[
\phi_{L\otimes M}\vert_{\Sigma_{\sigma\tau}}\colon  \Sigma_{\sigma\tau}\dra \P^5
\]
is a surface.
\end{enumerate}
\end{claim}

\begin{proof}
To prove~\ref{enum:Sigma1}, we note that the surface $\Sigma_{\sigma\tau}$ has class $\lll\mm$ and its associated matrix (as in~\eqref{eqmatriceO}) is
\[
M_{[\Sigma_{\sigma\tau}]}=\begin{pmatrix}
0 & 2 \\
2 & 0
\end{pmatrix}.
\]
If $\Sigma_{\sigma\tau}$ is not irreducible or not reduced, there   exist surfaces $\Sigma_1$, $\Sigma_2$ in $X$, such that
\begin{equation*}
M_{[\Sigma_1]}=M_{[\Sigma_2]}=\begin{pmatrix}
0 & 1 \\
1 & 0
\end{pmatrix}.
\end{equation*}
We show that no such   decomposition  exists under our assumptions.\

Since $(X,\lll,\mm)$ is very general, we may apply Proposition~\ref{prop:SurfacesOnHK4} and write
\[
[\Sigma_i]=t_{i} \lll^2+u_{i}\lll\mm+v_{i} \mm^2+w_{i} q_X^\vee
\]
with $t_{i}, u_i,v_i\in \Q_{\ge0} $ (because $\lll$ and $\mm$ are nef) and $w_i\in \Q $.\
Using \eqref{eqdematl2m2}, we find
\begin{equation*}\label{eqpoureta3eta4}
t_{i}=v_i=0,\quad 2u_{i}+25w_{i}=1,
\end{equation*}
so that
\[
[\Sigma_i]=\frac{1}{2}\Bigl(1-25w_{i}\Bigr) \lll\mm+w_{i} q_X^\vee.
\]
We obtain $w_2=- w_1$, so we can assume $w=:w_1\geq 0$.\
We have
\begin{equation*}\label{eqSigmai}
[\Sigma_i]=\frac{1}{2}\lll\mm-(-1)^i w\Bigl(q_X^\vee-\frac{25}2\lll\mm\Bigr),\quad \textnormal{for }i\in\{1,2\}.
\end{equation*}
Using~\eqref{qxc}, we get
\begin{equation*}\label{eqnumbres}
\Sigma_1^2= \Sigma_2^2=\frac{1}{2}+525w^2,\qquad \Sigma_1\Sigma_2=\frac{1}{2}-525w^2.
\end{equation*}
Since these numbers are integers, $w$ is positive and $2\cdot w^2=1\cdot 3\cdot 7\cdot (5w)^2$ is an integer, which implies that the denominator of $w$ is at most $5$, so that $w\geq \frac{1}{5}$.\
Then,
\begin{equation}\label{eqeffectivec2}
w q_X^\vee=[\Sigma_1]+\frac{1}{2}(25w -1)\lll\mm
\end{equation}
with $25w -1>0$.\ But this is impossible: by Proposition~\ref{lem:NefCone4folds} and~\cite[Proposition 3.2]{huyKahler}, the closure of the K\"ahler cone of $X$ is the subset $\overline{\cK}_X\subset H^{1,1}(X,\R)$ consisting of those real $(1,1)$-classes $\omega$ satisfying $q_X(\omega)\geq 0$, $q_X(\omega,\omega_0)>0$ for a K\"ahler class $\omega_0$, and $q_X(\omega,\lll)\ge0$ and $ q_X(\omega,\mm)\geq0$.
For all $\omega\in \overline{\cK}_X$ and all effective classes $\eta\in{\Hdg}^4(X,\Q)$, we have $\int_X\eta\omega^2\geq 0$.
Furthermore, for all $\omega\in \overline{\cK}_X$ such that $q_X(\omega)=0$, we have $\int_X q_X^\vee \omega^2=0$; so by~\eqref{eqeffectivec2} we have
\[
0=\int_{\Sigma_1}\omega^2+\frac{1}{2}(25w -1)\int_X\lll\mm\omega^2,
\]
where both terms on the right are nonnegative.\
Hence $\int_X\lll\mm\omega^2=0$; this implies that the class~$\lll\mm$ is proportional to $q_X^\vee$, which is absurd.

Claim~\ref{enum:Sigma2}, namely the fact that the image $\phi_{L\otimes M}(\Sigma_{\sigma\tau})$ is a surface, is proved as follows.\
If $\phi_{L\otimes M}(\Sigma_{\sigma\tau})$ is a linearly nondegenerate curve in $\P^4$, it has degree~$d\ge4$.\ Let $\fff$ be the class of a fiber of $\Sigma_{\sigma\tau}\dra\phi_{L\otimes M}(\Sigma_{\sigma\tau})$.\
 We have
\[
(\lll+\mm)[\Sigma_{\sigma\tau}]=(\lll+\mm)\lll\mm = \lll^2\mm + \lll\mm^2
\]
and this class can be written as $d\fff+\eee$, where $\eee$ is an effective curve class.\  This contradicts the fact that the class of any effective curve in $X$ has the form $t \lll^2\mm+u \lll\mm^2$  with   $t, u\in\frac12\Z_{\ge0}$.\end{proof}
The claim gives us a contradiction, since we know by~\eqref{eqineq4} that $\phi_{L\otimes M}(\Sigma_{\sigma\tau})$ is a linearly nondegenerate surface in $\P^4$, while, by~\eqref{eqineq3}, it is contained in at least~$\binom{4+2}{2}-7=8$
independent quadrics; this contradicts Castelnuovo's lemma, which says that it is contained in at most $\binom{2+1}{2}$ independent quadrics.
\end{proof}

\subsection{Case where either $L$ or $M$   has no nonzero sections}\label{zero}
We keep the same hypotheses:~$(X,\lll,\mm)$ is a very   general triple satisfying~\eqref{hypo} and we are in case~\ref{caseC1}.\ This section is devoted to the proof of the following result.

\begin{prop}\label{prop:CaseB'2}
If  either $H^0(X,L)=0$ or $H^0(X,M)=0$, the image of the rational map $\phi_{L\otimes M}\colon X\dra \P^5$ is rationally connected.
\end{prop}

Our aim is to prove that
 \[
Y\coloneqq\Im(\phi_{L\otimes M})\subset \P^5  
\]
is rationally connected, that is, that a general pair of points of $Y$ are joined by a rational curve.\ 
If \mbox{$\dim (Y)<4$,} this holds  by \cite[Theorem~1.4]{linthese}, so it suffices to prove Proposition~\ref{prop:CaseB'2} in  the case $\dim (Y)=4$, that is, when $Y\subset \P^5$ is a hypersurface.\

 We will assume  $H^0(X,L)=0$; the case  $H^0(X,M)=0$ is of course analogous (just permute~$L$ and $M$).\ 

\begin{prop}\label{prop6.3precise} If   $H^0(X,L)=0$ and $\dim (Y)=4$, a plane section $C_{0,W_3}$ of $Y$ defined by a general vector subspace $W_3\subset H^0(X,L\otimes M)$ of dimension $3$   is either  of geometric genus~$0$ or  a smooth cubic curve.
\end{prop}

Proposition \ref{prop6.3precise} implies Proposition~\ref{prop:CaseB'2} as follows.\  Proposition~\ref{prop6.3precise} then  shows that either~$Y$ is  rationally connected since  its general plane sections are rational curves, or it is  a cubic hypersurface, hence it is uniruled.\ In the second case, we can   apply \cite[Theorem~1.4]{linthese} again to the maximal rationally connected quotient of $Y$, which has dimension $<4$.\ It is thus rationally connected and is therefore a point by~\cite{GHS}, proving again that $Y$ is rationally connected.

We   start the proof of Proposition~\ref{prop6.3precise} by introducing some notation.\
Set
\begin{align*}
&W_6\coloneqq H^0(X,L\otimes M),\\
&W_{10}\coloneqq H^0(X,L^{\otp 2}\otimes M),\\
&W_{36}\coloneqq H^0(X,L^{\otp 3}\otimes  M^{\otp 2}),
\end{align*}
with multiplication map
\[
\mu\colon W_6\otimes W_{10}\lra W_{36}.
\]

\begin{lemm}\label{lem:NoRk2}
If $H^0(X,L)=0$,  there are no rank-2 elements in $\Ker(\mu)$.
\end{lemm}

\begin{proof}
A rank-$2$ element in~$\Ker(\mu)$ is given by a nontrivial relation
$\alpha\sigma=\beta \tau $ in $H^0(X,L^{\otp 3}\otimes M^{\otp 2})$ with $\alpha,\,\beta\in H^0(X,L\otimes M)$ linearly independent.\ Since $H^0(X,L)=0$, any divisor in $|L\otimes M|$ is irreducible and reduced, hence the divisors of $\alpha$ and $\beta$  have no common component.\ It follows that $\tau\in H^0(X,L^2\otimes M)$ vanishes on the divisor of $\alpha$ hence can be written as the product of~$\alpha$ by a section of $L$.\ This implies $\tau=0$, which is absurd.
\end{proof}

\begin{lemm}\label{leW6W10}
Assume there are no rank-$2$ elements in $\Ker(\mu)$.\
For a general $3$-dimensional vector subspace $W_3\subset W_6$, the restriction $\mu\vert_{W_3}\colon W_3\otimes W_{10}\to W_{36}$ of the map $\mu$ has rank $\geq 28$.
\end{lemm}

\begin{proof}
Let $\cS_3\to \Gr(3,W_6)$ be the rank-$3$ tautological subbundle.\
The natural sheaf inclusion $\cS_3\to  W_6\otimes \cO_{\Gr(3,W_6)}$ induces a morphism $f\colon \P(\cS_3\otimes W_{10})\to\P(W_6\otimes W_{10})$ which makes $\P(\cS_3\otimes W_{10})$ a smooth birational model of the set of elements of rank at most $ 3$ in $\P(W_6\otimes W_{10})$.\
We set
\[
R_3\coloneqq f^{-1}(\P(\Ker(\mu))).
\]
Since, by assumption, there are no rank-$2$ elements in $\Ker(\mu)$, the scheme $R_3$ is isomorphic to the set of elements of rank $\leq 3$ in $\P(\Ker(\mu))$.

The fiber of the natural morphism $\pi\colon R_3\to\Gr(3,W_6)$ over a point $[W_3]$ is the space $\P(\Ker(\mu\vert_{W_3}))$.\
Arguing by contradiction, if the conclusion of the lemma does not hold, the fiber of $\pi$ has dimension $\geq 2$, hence $\dim (R_3)\geq 11$.

Let $H$ be  the restriction to $R_3$ of the line bundle $\cO_{\P(\cS_3\otimes W_{10})}(1)$.\
We have a tautological inclusion $H^{-1}\to \pi^* \cS_3\otimes W_{10}$
\begin{equation}\label{eqmorphdual}
W_{10}^\vee\otimes \cO_{R_3}\lra  H\otimes \pi^*\cS_3.
\end{equation}
Since there are no rank-$2$ elements in $\Ker (\mu)$,  the morphism~\eqref{eqmorphdual}  is surjective, with kernel $\cK $ a locally free sheaf  of rank $7$.\
Thus we have $c_i(\cK)=0$ for $i>7$, that is, $s_i(H\otimes \pi^*\cS_3)=0$, for all $i>7$, where the $s_i$ denote the Segre classes.

We now deduce a contradiction.\
For any line bundle $H$ on a variety and any vector bundle~$\cE$ of rank $3$, we have the relation
\begin{eqnarray}\label{eqrel}
s_i(\cE\otimes H)=\sum_{j=0}^i (-1)^j \binom{i+2}{j} H^js_{i-j}(\cE).
\end{eqnarray}
In our case, one has $s_j(\cS_3)=0$ for $j\geq 4$, because of the exact sequence
\[
0\to   \cS_3\to  W_6\otimes \cO_{\Gr(3,W_6)}\to  \cQ_3\to  0
\]
on the Grassmannian, where $\cQ_3$ is locally free of rank $3$.\
The relations \eqref{eqrel}  thus give four linear relations involving  $s_0(\pi^* \cS_3),\ldots, s_3( \pi^* \cS_3)$, namely
\begin{align*}\label{eqrelexpl}
  0&=s_8(\pi^* \cS_3\otimes H  )\nonumber\\
  &= \binom{10}{8}H^8s_0(\pi^* \cS_3)-\binom{10}{7}H^7s_1(\pi^* \cS_3)+\binom{10}{6}H^6 s_2(\pi^* \cS_3)-\binom{10}{5}H^5s_3(\pi^* \cS_3),\\
  \nonumber
 0&=s_9(\pi^* \cS_3\otimes H  )\\
 &= -  \binom{11}{9}H^9s_0(\pi^* \cS_3)+\binom{11}{8}H^8s_1(\pi^* \cS_3)-\binom{11}{7}H^7s_2(\pi^* \cS_3)+\binom{11}{6}H^6s_3(\pi^* \cS_3), \\
  \nonumber
 0&=s_{10} (\pi^* \cS_3\otimes H  )\\
 &=\binom{12}{10}H^{10}s_0(\pi^* \cS_3)-\binom{12}{9}H^9s_1(\pi^* \cS_3)+\binom{12}{8}H^8s_2(\pi^* \cS_3)-\binom{12}{7}H^7s_3(\pi^* \cS_3), \\
  \nonumber
 0&=s_{11} (\pi^* \cS_3\otimes H  )\\
 &=-\binom{13}{11}H^{11}s_0(\pi^* \cS_3)+\binom{13}{10}H^{10}s_1(\pi^* \cS_3)-\binom{13}{9}H^9s_2(\pi^* \cS_3)+\binom{13}{8}H^8s_3(\pi^* \cS_3).
\end{align*}

We can assume $\dim(R_3)=11$, replacing it by a proper algebraic subset if necessary.\
Multiplying these equations by   adequate powers of $H$, we get the linear relations
\begin{equation*} 
\begin{split}
& 0=H^3s_8(\pi^* \cS_3\otimes H),\qquad 0=H^2s_9(\pi^* \cS_3\otimes H  ),\\
& 0=H s_{10} (\pi^* \cS_3\otimes H  ),\qquad 0=s_{11} (\pi^* \cS_3\otimes H),
\end{split}
\end{equation*}
which, by expanding as above,  give four linear relations between the four intersection numbers
\begin{equation*} \label{eqnumbers}
H^8s_3(\pi^* \cS_3),\ H^9s_2(\pi^* \cS_3),\ H^{10}s_1(\pi^* \cS_3), \ H^{11}
\end{equation*}
on $R_3$.\
Since these four  linear relations are clearly independent, we conclude that these four intersections numbers vanish, which contradicts the fact that $H$ is ample on $R_3$.
\end{proof}

Lemma~\ref{lem:NoRk2} and Lemma~\ref{leW6W10} together imply the following result.

\begin{coro}\label{lealternative}
Assume $H^0(X,L)=0$.\ For a general $3$-dimensional vector subspace $W_3\subset W_6$, the image $\mu(W_3\ot  W_{10})$ has dimension $\geq 28$.
\end{coro}

\begin{proof}[Proof of Proposition~\ref{prop6.3precise}.]
We choose a general $3$-dimensional vector subspace $W_3\subset W_6$.\
The locus
defined by the vanishing of the sections in $W_3$ has a mobile part $C_{W_3}\subset X$ which, by Bertini's theorem, is an irreducible curve
which   dominates the  plane curve $C_{0,W_3}$ via $\phi_{L\otimes M}$.\ Consider the restriction maps
\[
r^{p,q}\colon H^0(X,L^{\otp p}\otimes M^{\otp q})\lra  H^0(C_{W_3},(L^{\otp p}\otimes M^{\otp q})\vert_{C_{W_3}})
\]
for $p,q>0$.\
Let us set $W_{p,q}'\coloneqq \Im(r^{p,q})$.\
We will estimate the dimension of $W_{p,q}'$, for $(p,q)\in\{(1,1),(2,1), (3,2)\}$.\
First of all, we note that $W_{1,1}'$ has dimension~3.

We then claim that $W_{3,2}'$ has dimension $\leq 8$.
Indeed, we have the inclusion
\[
\mu(W_3\ot W_{10})\subset \Ker (r^{3,2})
\]
and, by Corollary~\ref{lealternative},   the space on the left has dimension at least $28$, while $h^0(X,L^{\otp 3}\otimes M^{\otp 2})=36$.

Finally, we claim that
\begin{enumerate}[{\rm (a)}]
\item\label{enum:claim1_part1} either the space $W_{2,1}'$ has dimension $\geq 5$,
\item\label{enum:claim1_part2} or it has dimension $4$ and the image $\phi_{L^{\otp 2}\otimes M}( C_{W_3})$ is a rational normal cubic curve in~$\P^3$.
\end{enumerate}
Indeed, since $\phi_{L\otimes M}$ is a dominant rational map to a hypersurface $Y\subset \P^5$, the curve $C_{W_3}$ can be chosen to pass through a  general triple of points $x,y,z\in X$.\
Assume  $\phi_{L^{\otp 2}\otimes M}( C_{W_3})$ spans at most a $\P^3$ in $\P^9=\P(W_{10}^\vee)$.\
This $\P^3$ then contains the projective plane  spanned by the images of $x$, $y$,  $z$, and thus this plane intersects the curve $\phi_{L^{\otp 2}\otimes M}(C_{W_3})$ in at least a fourth point, unless   $\phi_{L^{\otp 2}\otimes M}(C_{W_3})$ is a rational normal curve of degree $3$.\

In the former case, we conclude that the variety $Y'\coloneqq\phi_{L^{\otp 2}\otimes M}(X)$, which spans $\P^9$, has the property that a general trisecant plane of $Y'$ is $4$-secant.\
This is absurd, since a general $1$-dimensional linear section of $Y'$ spans at least a $\P^6$.\

In the latter case, the curve $\phi_{L^{\otp 2}\otimes M}(C_{W_3})$ has  degree $\leq 3$.\
This curve cannot be a plane curve, since otherwise the projection of $Y'$ through any of its points $y$ would contain a line through any two points $x$, $z$, hence would be a projective space of dimension at most $4$.\
Hence it must be a rational normal curve in $\P^3$, as claimed.

Let us now consider the multiplication map
\[
W_{1,1}'\otimes W_{2,1}'\lra  W_{3,2}'.
\]
In case~\ref{enum:claim1_part1}, the three spaces have respective dimensions $3$, at least $5$, and at most $8$.\
Since the curve $C_{W_3}$ is irreducible, we can apply Lemma \ref{lepourconcprop72} below: it says that   this is possible only if the linear system $W_{1,1}'$ factors through an elliptic plane  curve  or  a rational curve.

In case~\ref{enum:claim1_part2}, we have $W_{2,1}'=\Sym^3(W_2'')$, for some $2$-dimensional linear system $W_2''$ on~$C_{W_3}$.\
We now consider the multiplication maps
\begin{align*}
W_{1,1}'\otimes W_2''&\lra  W_{k_1}'',\\
W_{k_1}''\otimes W_2''&\lra  W_{k_2}'',\\
W_{k_2}''\otimes W_2''&\lra  W_{3,2}',
\end{align*}
where $W_{k_1}''$ and $W_{k_2}''$ are spaces of sections of adequate line bundles on~$C_{W_3}$, of respective dimensions $k_1$ and $k_2$.\
The Hopf lemma (\cite[p.~108]{acgh}) gives
\begin{align*}
\dim (W_{k_1}'')&\geq \dim (W_{1,1}')+1,\\
\dim (W_{k_2}'')&\geq \dim (W_{k_1}'')+1,\\
\dim (W_{3,2}')&\geq \dim (W_{k_2}'')+1.
\end{align*}
As $\dim (W_{1,1}')=3$ and $\dim (W_{3,2}')\leq 8$, the three inequalities above cannot all be strict, hence one of them must be an equality.\
It is well known that this implies  that the plane curve $C_{0,W_3}=\phi_{L\otimes M}(C_{W_3}) $ is rational, thus completing the proof of Proposition \ref{prop6.3precise}.
\end{proof}

We used  above the following lemma, for which we could not find a reference.

\begin{lemm}\label{lepourconcprop72}
Let $C$ be a smooth connected projective curve and let $H$ and $H'$ be  line bundles on $C$.\ Let
$W_3\subset H^0(C,H)$ and $W_k\subset H^0(C,H')$ be  base-point-free linear systems on $C$, of respective linear dimensions $3$ and $k$, with $k\geq 4$.\
Assume that the rank of the multiplication map
$$\mu\colon W_3\otimes W_k\lra H^0(C,H\otimes H')$$
is at most $3+k$.\ Then, both linear systems factor through a morphism $C\to C_0$, where $C_0$ is either a  rational curve or a degree-$3$ elliptic curve in $\P(W_3)$.
\end{lemm}

\begin{proof} Denote by $\phi_3\colon C\to \P^2$   the morphism  induced by $W_3$ and by
$\phi_k\colon C\to \P^{k-1}$ the morphism  induced by $W_k$.\   We first claim that the result is true if  $\phi_{k}$ does not factor through~$\phi_3$.\ Indeed, assume there is a length-$2$ subscheme
$z\subset C$ that imposes only one  condition on $W_3$ and two conditions on $W_k$.\ Then
$z$ imposes two conditions on $\Im(\mu)$, hence the multiplication map
$$\mu_z\colon W_3(-z)\otimes W_k\lra   H^0(C, H\otimes H')$$
has rank at most $k+1$, while $W_3(-z)$ has dimension $2$.\ We are thus in the  equality case of the Hopf lemma  and we conclude that $W_k$ is the pullback to $C$ of $H^0(\P^1,\cO_{\P^1}(k-1))$ via a morphism
$\psi\colon C\to  \P^1$.\ This case is easily concluded by studying  the
multiplication maps
$$W_3\otimes H^0(\P^1,\cO_{\P^1}(k-2))\lra H^0(C,H\otimes \psi^*\cO_{\P^1}(k-2)),$$
 with image $W'$, and
$$W'\otimes H^0(\P^1,\cO_{\P^1}(1))\lra H^0(C,H\otimes \psi^*\cO_{\P^1}(k-1)), $$
of rank $\leq 3+k$.\
By the Hopf lemma applied to both maps,   $\dim (W')\in\{k+1,k+2\}$.\ Thus, one the two maps  satisfies  the equality in the Hopf lemma, hence    $\phi_3$ also factors through $\psi$.\ This proves the  claim.

We now prove a similar claim with $\phi_3$ and $\phi_{k}$ permuted.\ Assume  there is a length-$2$ subscheme $z\subset C$ that imposes only one condition on $W_k$ but two conditions on $W_3$.\ This produces a $(k-1)$-dimensional vector subspace $W_{k}(-z)\subset W_k$ of sections vanishing on $z$, with the property that the multiplication map
$$\mu_z\colon W_3\otimes W_{k}(-z)\lra H^0(X,H\otimes  H'(-z))$$
has rank $\leq k+1$, which is the minimum allowed by the Hopf Lemma.\ We conclude that there is a morphism $\psi:C\rightarrow \P^1$ such that $W_3=\psi^*H^0(\P^1,\cO_{\P^1}(2))$.\
We now consider the   multiplication maps
$$\mu'_1\colon  W_2\otimes  W_k\lra H^0(C, H'\otimes \psi^* \cO_{\P^1}(1)),$$
with image $W'$, and
$$ \mu'_2\colon W_2\otimes  W'\lra H^0(C, H'\otimes \psi^* \cO_{\P^1}(2))=H^0(C,H\otimes H').$$
We know that  $\mu'_2$ has rank at most $3+k$, so either $\dim (W')=k+1$ and $\mu'_1$ satisfies  the equality case of  the Hopf lemma, or
$\dim (W')=k+2$ and  $\mu'_2$ satisfies  the equality case of  the Hopf lemma.\ In both cases, we conclude that both linear systems factor through $\psi$.

Using these two claims, we can now assume that both linear systems factor through  the curve $C_0\coloneqq\phi_3(C)\subset \P(W_3)$,  that $C_0$ is birationally isomorphic to its image in $\P(W_k)$, and that  the normalization  $C'_0$ of $C_0$ is not rational, as otherwise the lemma is proved.\ We  denote by~$H_0$ and $H'_0$ the line bundles on  $C'_0$  whose respective pullbacks to $C$ are $H$ and $H'$.\      For    general points $x_1,\ldots,x_{k-3}\in C'_0$, we have a
$3$-dimensional space $W'_3
\subset W_k$ of sections vanishing at $x_1,\ldots,x_{k-3}$, and the multiplication map
$$\mu'\colon  W_3\otimes W'_3\lra H^0(C'_0, H_0\otimes H'_0)$$
has rank $\leq 6$ since its image is contained in $ \Im (\mu)$ and vanishes at $x_1,\dots,x_{k-3}$.
We can thus assume that the kernel
\begin{equation}
\label{eqKsubset}
K\coloneqq  \Ker (\mu')\subset  W_3\otimes W'_3
\end{equation}
has dimension $3$ (if it has dimension $4$,   we are in the equality case of the Hopf lemma, so we can ignore this case).\
The inclusion \eqref{eqKsubset} induces a morphism
\begin{equation}
\label{eqmorphKW}
K\otimes \cO_{\P(  W_3)}\lra  W'_3\otimes  \cO_{\P(  W_3)}(1)
\end{equation}
of rank-$3$ vector bundles
on $\P(W_3)$.\ This morphism  has rank at most $ 2$ along $C_0$ and   must have rank at least $2$ generically on $\P(W_3)$ since otherwise, the image would be  a rank-$1$ sheaf with at least three independent sections contained in $W'_3\otimes  \cO_{\P(  W_3)}(1)$, hence a copy of  $\cO_{\P(  W_3)}(1)$ and the elements of~$K$ would be rank-$1$ tensors in $W_3\otimes W'_3$.\   Since we assumed that the curve $C'_0$ is not rational, no quadratic equation vanishes on it, hence  the morphism \eqref{eqmorphKW} also has generic rank $2$ along the curve $C_0$.

If the curve  $C_0$ has degree at most $ 3$,  the lemma is proved.\ Otherwise, the morphism~\eqref{eqmorphKW}  has   rank at most $ 2$ everywhere on $\P(  W_3)$ and is generically of rank $2$ along $C_0$, which means that the three polynomials of type $(1,1)$ on $\P( W_3)\times \P(  W'_3)$ given by $K$ vanish on  a surface $\Sigma$ which contains the natural  embedding of $C'_0$ in $\P( W_3)\times \P(  W'_3)$ and is birationally isomorphic to  $\P(  W_3)$ by the first projection.\ The surface $\Sigma$  is  also birationally isomorphic  to  $\P( W'_3)$ by the second projection since its image in $\P( W'_3)$ contains the image of $C'_0$ in $\P( W'_3)$ which, being birationally isomorphic to $C'_0$ for a generic choice of $x_1\ldots,x_{k-3}$, is also not rational.\

We claim that the surface $\Sigma\subset \P(  W_3)\times \P( W'_3)$ is  the graph of an isomorphism $\P(  W_3)\isom \P( W'_3)$.\ Indeed, as $\Sigma $ is  contained in three hypersurfaces of type $(1,1)$, it is an irreducible component of the complete intersection of two such hypersurfaces, but it is not the complete intersection of two such hypersurfaces (since there is a third equation of type $(1,1)$ vanishing on it).\  It follows that the class of $\Sigma$ is of the form $(h_1+h_2)^2-e=h_1^2+2h_1h_2+h_2^2-e$, where  $h_i\coloneqq c_1({\rm pr}_i^*\cO_{\P^2}(1))$ and the class $e$  is effective and nonzero on $\P(  W_3)\times \P( W'_3)$.\ Since the   projections ${\rm pr}_1$ and ${\rm pr}_2$ are dominant and $(\Sigma\cdot h_1\cdot h_2)^2\ge (\Sigma\cdot h_1^2)(\Sigma\cdot h_2^2)$ (Hodge Index Theorem), the only possibility is   $[\Sigma]=h_1^2+h_1h_2+h_2^2$, that is, $[\Sigma]$ is the class of the graph of an isomorphism $\P(  W_3)\cong \P( W'_3)$.\ This implies that $\Sigma$ itself is the graph of an isomorphism since we then have $\Sigma^*\cO_{\P( W'_3)}(1)= \cO_{\P( W_3)}(1)$.\ This proves the claim.\
The claim   implies that the
line bundles  $H_0$ and $H'_0(-x_1-\dots-x_{k-3})$ on $C'_0$ coincide.\ As $x_1\ldots,x_{k-3}$ are general points of   $C'_0$ and $k\geq 4$, this implies that $C'_0$ is rational, which is a contradiction.
\end{proof}

\subsection{Proof of Proposition~\ref{prop:CaseC1}}\label{proof}
If $H^0(X,L) $ and $H^0(X,M) $ are both nonzero, by Lemma~\ref{lecasholhoM} and Lemma~\ref{lem:H0large}, after possibly permuting $L$ and $M$, the line bundle $L$ is globally generated and thus gives a Lagrangian fibration $f\colon X\to \P^2$ with $f^*\cO_{\P^2}(1)\isom L$.\ So we are in case \ref{prop64a} of Proposition~\ref{prop:CaseC1}.

If either $H^0(X,L) $ or $H^0(X,M) $ is zero, the image of the rational map~$\phi_{L\otimes M}$ is rationally connected by Proposition~\ref{prop:CaseB'2} and we are in case \ref{prop64b} of Proposition~\ref{prop:CaseC1}.\


\section{The divisorial contraction case}\label{sec:DivisorialContraction}

Let $(X,\lll,\mm)$ be a very general triple satisfying~\eqref{hypo}.\
By Proposition~\ref{prop:SurfacesOnHK4}, we have  $\NS(X)=\Z\lll\oplus\Z\mm$.\
We assume in this section that we are in case~\ref{caseC2}: the class $\lll$ is nef, while the class~$\mm$ is isotropic and not nef.\ Thus, we have  $\Mov(X)=\R_{\ge 0}\lll+\R_{\ge 0}(\lll+\mm)$ and there is a divisorial contraction  $c\colon X\to Y$ defined by  some positive power of the semi-ample line bundle $L\otimes M$, whose exceptional locus is the prime divisor $E$ with class $-\lll+\mm$.\
Our aim is to prove Proposition~\ref{prop:CaseC2}: some tensor power of $L$ defines a Lagrangian fibration on $X$.

The fourfold $Y$ is Gorenstein with trivial canonical sheaf, rational singularities, and singular locus the surface $\Sigma\coloneqq c(E)$.\
The class in $H_2(X,\Z)\isom H^2(X ,\Z)^\vee$ of a general fiber of $E\to \Sigma$ is given by the linear form $q_X(-\lll+\mm,\bullet)$ which, since   $q_X(-\lll+\mm,\lll)=1$, is nondivisible.\
The class of any curve contracted by $c$ is a multiple of that class, hence all $1$-dimensional fibers of $c$ are irreducible, smooth rational curves.

\begin{prop}\label{prop:NoFibersDimension2}
The fibers of the contraction $c\colon X\to Y$ all have dimension at most~$ 1$, the varieties $E$ and $\Sigma$ are smooth, and the restriction $c_E:=c\vert_E\colon E\to \Sigma$ is a $\P^1$-fibration.
\end{prop}

\begin{proof}
We argue by contradiction and assume that there is an integral surface $S\subset X$ such that $c(S)$ is a point in $Y$.\ Then $S\subset E$.\ We first claim that the class
\begin{equation}\label{eq:S'}
[S'] \coloneqq 2(\lll+\mm)(-\lll+\mm) - [S]
\end{equation}
in ${\Hdg}^4(X,\Z)$ is effective.

Indeed, by assumption, the line bundles $L^2\otimes M^2$ and $L^3\otimes M$ are big and nef on $X$.\ By the Kawamata--Viehweg vanishing theorem and~\eqref{hypoprime}, we get
\[
h^0(X, L^2\otimes M^2)=21\quad\text{ and }\quad h^0(X, L^3\otimes M)=15.
\]
It follows that the restriction map
\[
H^0(X,L^2\otimes M^2)\rightarrow H^0(E,(L^2\otimes M^2)\vert_E),
\]
whose kernel equals $H^0(X,L^3\otimes M)$, has rank at least~$6$.\
On the other hand, as $L\otimes M$ is numerically trivial on $S$, the restriction map
\[
H^0(X,L^2\otimes M^2)\rightarrow H^0(S,(L^2\otimes M^2)\vert_S)
\]
has rank at most~$1$.\
It follows that  there exists a section of $L^2\otimes M^2$ vanishing on $S$ which does not vanish identically on $E$.\
Since the class of $E$ is $-\lll+\mm$, the claim follows.

We now compute the  intersection matrices $M_{[S]}$ and $M_{[S']}$ defined in~\eqref{eqmatriceO}.\ They are nonzero since there is an ample class  in $\Z\lll\op\Z\mm$.\
Since the surface $S$ is contracted by $c$, the line bundle $L\otimes M$ is numerically trivial on $S$ and we get
\[
M_{[S]}=
\begin{pmatrix}
t & -t \\
-t & t
\end{pmatrix}
\]
for some positive integer $t$.\
Hence, using~\eqref{eq:S'}, we get
\[
M_{[S']}=
\begin{pmatrix}
4-t & t \\
t & -4-t
\end{pmatrix}.
\]
Since $[S']$ is effective, we get $4-t\geq0$, hence $t\in\{1,2,3,4\}$.

By Proposition~\ref{prop:SurfacesOnHK4}, we can write
\begin{equation}\label{eq:Paris20211103}
\begin{split}
    &[S] = \frac t2\, \Big( \lll^2 -\lll\mm + \mm^2 \Big) +   w\,\Big(q_X^\vee - \frac{25}{2} \lll\mm \Big), \\
    &[S'] = -\Big(2+ \frac t2 \Big)\lll^2 + \frac t2 \lll\mm + \Big(2 - \frac t2 \Big)\mm^2 -   w\,\Big(q_X^\vee - \frac{25}{2} \lll\mm \Big),
\end{split}
\end{equation}
for some $w\in\Q$.\
Since $[S]$ is an integral class, we have $[S]^2\in\Z$.\
Using~\eqref{qxc}, we get
\[
2[S]^2 =   t^2 + 525\, w^2=t^2 + 3\cdot  7\, (5w)^2,
\]
so that $5w\in\Z$.

From~\eqref{eq:Paris20211103}, by the same reasoning used at the end of the proof of Lemma~\ref{lecasholhoM}, for any class $\omega\in \overline{\cK}_X\subset H^{1,1}(X,\R)$ (in our situation this means in particular that $q_X(\omega,\lll)\ge0$ and $q_X(\omega,-\lll+\mm)\geq0$) with $q_X(\omega)=0$, we obtain
\begin{equation}\label{eq:Paris20211103_second}
\begin{split}
    & 0\leq \int_X [S]\omega^2 = \frac t2 \int_X\lll^2\omega^2- \Big(\frac {t+25w}2\Big)\int_X \lll\mm\omega^2 + \frac t2\int_X \mm^2\omega^2,  \\
    & 0\leq \int_X [S']\omega^2 = - \Big(2+\frac t2\Big)\int_X \lll^2\omega^2
     + \Big(\frac {t+25w}2\Big)\int_X \lll\mm\omega^2
     +\Big(2-\frac t2\Big)\int_X \mm^2\omega^2 .
\end{split}
\end{equation}
Moreover, since $q_X(\omega)=0$ and $c_X=3$,   the Fujiki relation~\eqref{eq:Fujiki3} implies, for all $\alpha,\beta\in H^2(X,\Z)$,
\[
\int_X \alpha\beta\omega^2 = 2 q_X(\alpha,\omega)\, q_X(\beta,\omega).
\]
Thus, from~\eqref{eq:Paris20211103_second}, we deduce
\[
\frac {t+25w}2  \leq t\quad\text{ and }\quad \frac {t+25w}2  \geq t,
\]
and thus
\[
25 w = t, \qquad \text{ with } t\in\{1,2,3,4\}\text{ and }5w\in\Z,
\]
which is impossible.\ This proves that $c$ contracts no surfaces.

The other statements of the lemma then follow from
\cite[Theorem~1.3(ii)]{wie}.
\end{proof}

Since the line bundle $L$ has intersection~1 with all fibers of $c_E$, we immediately deduce the following result.

\begin{lemm}\label{lem:EisP1bundle}
Let $\cE\coloneqq c_{E\,*}L$.
We have an isomorphism $E\cong\P_{\Sigma}(\cE^\vee)$ over $\Sigma$.
\end{lemm}

By Lemma~\ref{lem:EisP1bundle}, the line bundle $(L\otimes M)\vert_E$ descends to a line bundle $H_{\Sigma}$ on~$\Sigma$, which is ample since some positive power of $L\otimes M$ is the pullback of an ample line bundle on $Y$.\
We will study  in more detail in Section~\ref{sec:OGrady} the surface $\Sigma$, the polarization $H_{\Sigma}$, and the rank-2 vector bundle $\cE$.
For the moment, we just use their existence to finish the proof of Proposition~\ref{prop:CaseC2}.
We need one last result before that.

\begin{lemm}\label{levanh2}
One has $H^2(X, M^{\otp 2})=H^4(X, M^{\otp 2})=0$ and $h^0(X, M^{\otp 2})\ge3$.
\end{lemm}

\begin{proof}
We first prove $H^2(X, M^{-2})=0$.\
Consider the exact sequence
\begin{equation*}
0\to \cO_X(-E)\to \cO_X\to \cO_E\to 0.
\end{equation*}
Tensoring by $(L\otimes M)^{\otp -1}$, we obtain
\begin{equation*}
0\to  M^{-2}\to (L\otimes M)^{\otp -1}\to (L\otimes M)^{\otp -1}\vert_E\to 0.
\end{equation*}
Since $H^2(X,(L\otimes M)^{\otp -1})=H^2(X,L\otimes M)=0$, it  suffices to prove  $H^1(E,(L\otimes M)^{\otp -1}\vert_E)=0$.\
But, by Lemma~\ref{lem:EisP1bundle}, we have
\[
H^1(E,(L\otimes M)^{-1}\vert_E)\isom H^1(\Sigma, H_{\Sigma}^{-1})=0,
\]
since $H_{\Sigma}$ is ample on the smooth surface $\Sigma$, as we wanted.

By Serre duality, we get $H^2(X,M^{\otp 2})=0$.\
We also have $H^4(X,M^{\otp 2})=H^0(X,M^{\otp -2})=0$.

Finally, the Riemann--Roch theorem takes the form
\begin{equation}\label{eq:Paris20211103_third}
h^0(X,M^{\otp  2})+h^2(X,M^{\otp  2})+h^4(X,M^{\otp2})\geq \chi(X,M^{\otp  2})=P_{RR,X}(0)=3,
\end{equation}
completing the proof of the lemma.
\end{proof}

In fact, the divisor $E$ is fixed in $|M^{\otp 2}|$, hence $h^0(X,M^{\otp  2})= h^0(X,L\otimes M)=6$, but we will use again the Riemann--Roch argument~\eqref{eq:Paris20211103_third} in the proof below.

\begin{proof}[Proof of Proposition~\ref{prop:CaseC2}.]
As observed in Section~\ref{subsec:Lagrangian4folds}, we only have to show  $h^0(X,L^{\otp 2})\geq 2$.\
This inequality follows from Lemma \ref{levanh2} by deformation and specialization: as we explained in Section~\ref{subsec:ExceptionalDivisor}, the reflection that permutes $\lll$ and $\mm$ is a monodromy operator.\
This means that one can deform the pair $(X,L)$ into the pair $(X,M)$ through a family $(\cX,\cL)\to T$.\
By semi-continuity, Lemma~\ref{levanh2} implies $h^2(\cX_t, \cL_t^{\otp 2})=h^4(\cX_t, \cL_t^{\otp 2})=0$ for $t\in T$ general.\
The Riemann--Roch argument of~\eqref{eq:Paris20211103_third} then implies $h^0(\cX_t, \cL_t^{\otp 2})\geq 3$ for $t\in T$ general.\
By upper semi-continuity again, we conclude $h^0(X,L^{\otp 2})\geq 3$, as we wanted.
\end{proof}


\section{Proofs of the main theorems}\label{sec:OGrady}

In this section, we prove Theorem~\ref{thm:MainThm} and Theorem~\ref{thm:SYZ}.\
Let $X$ be a  \hk\ fourfold.
We assume that there are classes $\lll,\mm\in H^2(X,\Z)$ with $\int_X \lll^4=0$ and $\int_X\lll^2\mm^2=2$.

The sublattice spanned by $\lll$ and $\mm$ is indefinite, hence its $q_X$-orthogonal has signature $(2,b_2(X)-4)$.\ 
By the surjectivity of the period map (\cite{HuyHK}), we can then deform $X$ and assume that  $(X,\lll,\mm)$ is a very general triple satisfying~\eqref{hypo} (hence also~(\ref{hypo}$^\prime$)).\
After possibly permuting~$\lll$ and $\mm$, we can further assume that $\lll$ is
nef (see~\eqref{cones}).\
By Corollary~\ref{cor:NoTwoIsotropicNef}, the class $\mm$ cannot be nef, and we are in the situation of Proposition~\ref{prop:CaseC2}.\
Let us denote by $f\colon X \to \P^2$ the Lagrangian fibration such that $f^*\cO_{\P^2}(1)\isom L^{\otp k_L}$, for some positive integer $k_L$.\
We start with the following vanishing result.

\begin{lemm}\label{lem:vanishing}
Under the above assumptions, let $p,q\in\Z$, with $q>0$.\
For all $i\geq 3$, we have
\[
H^i(X,L^{\otp p}\otimes M^{\otp q})=0.
\]
\end{lemm}

\begin{proof}
Choose $r\in\Z$ be such that $p+rk_L \geq q$.\
We can write
\[
L^{\otp p}\otimes M^{\otp q} =  L^{\otp p+rk_L}  \otimes L^{-rk_L} \ot M^{\otp q}=  L^{\otp p+rk_L} \otimes M^{\otp q}\otimes f^*\cO_{\P^2}(-r),
\]
where $ L^{\otp p+rk_L}\ot M^{\otp q}$ is big and nef.\
By \cite[Theorem~10.32]{kol} again, we obtain that $R^j f_* (L^{\otp p}\otimes M^{\otp q})$ vanishes for all $j>0$ and $f_* (L^{\otp p}\otimes M^{\otp q})$ is a vector bundle on~$\P^2$.\ 
In particular, this gives us what we need:
\[
H^i(X, L^{\otp p}\otimes M^{\otp q}) = H^i(\P^2, f_* (L^{\otp p}\otimes M^{\otp q})) = 0,
\]
for all $i\geq3$.
\end{proof}

We now study in more detail the surface $\Sigma$ and the vector bundle $\cE$.\
Recall that we have $(L\otimes M)\vert_E = c_E^* H_{\Sigma}$.\
Let us denote by $\hh\in \NS(\Sigma)$ the  class of $H_{\Sigma}$.

\begin{lemm}\label{lem:K3}
The pair $(\Sigma,\hh)$ is a polarized K3 surface of degree~2 with $\NS(\Sigma)=\Z \hh$.
\end{lemm}

\begin{proof}
By~\cite[Theorem~1.4]{wie}, we know that $\Sigma$ is a symplectic surface.\
To show that it is a K3 surface, it is enough to compute $\chi(\Sigma,\cO_\Sigma)=\chi(E,\cO_E)$.\
By the Riemann--Roch Theorem, we have
\[
\chi(E,\cO_E) = \chi(X,\cO_X)-\chi(X,\cO_X(-E))=3-1=2,
\]
as we wanted.\
Also,
\[
\hh^2 = \int_X (\lll+\mm)^2 (-\lll+\mm) \lll = 2.
\]
Thus, we are left to show that the N\'eron--Severi group of $\Sigma$ has rank one.

Let us consider the transcendental lattice $H^2(X,\Z)_{\rm tr}\coloneqq \NS(X)^\perp\subset H^2(X,\Z)$, which under our assumptions has rank~21.
For all $\alpha\in H^2(X,\Z)_{\rm tr}$,  we have, by~\eqref{eq:Fujiki3},
\[
\int_X \alpha(-\lll+\mm)(\lll+\mm)^2 = 0.
\]
Hence, we can write
\[
\alpha\vert_E = c_E^*\nn_{\alpha}
\]
for a unique class $\nn_{\alpha}\in \hh^\bot \subset H^2(\Sigma,\Z)$.
Moreover, again by~\eqref{eq:Fujiki3}, for all $\alpha,\beta\in H^2(X,\Z)_{\rm tr}$, we have
\[
q_X(\alpha,\beta) = \int_X \alpha\beta\lll(-\lll+\mm) = \int_{\Sigma}\nn_{\alpha}\nn_{\beta}.
\]
Therefore, the morphism
\[
\vartheta\colon H^2(X,\Z)_{\rm tr} \lra \hh^\bot \subset H^2(\Sigma,\Z), \qquad \alpha\longmapsto \nn_{\alpha}
\]
gives an isometry $\vartheta_{\C}\colon H^2(X,\Z)_{\rm tr}\otimes\C \isomto \hh^\perp\otimes\C$.

The morphism $\vartheta$ is a nonzero morphism of Hodge structures.\ Since the Hodge structure~$H^2(X)_{\rm tr}$
is irreducible,~$\vartheta$ is injective  hence $\Sigma$ has Picard number one.
 \end{proof}

\begin{lemm}\label{lem:bundle}
The vector bundle $\cE$ is the unique spherical stable bundle on $\Sigma$ with Mukai vector~$(2,H_\Sigma,1)$.
\end{lemm}

\begin{proof}
By Lemma~\ref{lem:K3}, the Picard group of $\Sigma$ is generated by $H_\Sigma$.\
Hence, we can write  the Mukai vector of $\cE$ as $v(\cE)=(2,s H_\Sigma, s')$, with $s,s'\in\Z$.\
By the Riemann--Roch Theorem, since $[E]=-\lll+\mm$, we have
\[
\chi(X,L (-E))=\chi(X,L^{\otp 2}\otimes M^{-1})=0.
\]
Hence, from the exact sequence
\begin{equation}\label{eq:Paris20211029}
0\to L (-E) \to L \to L \vert_E \to 0
\end{equation}
we get
\[
s'+2 = \chi(\Sigma,\cE) = \chi(E,L \vert_E) = \chi(X,L) - \chi(X,L (-E)) = \chi(X,L) = 3,
\]
and thus $s'=1$.
To compute $s$, we proceed similarly and use the exact sequence
\begin{equation}\label{eq:Paris20211028}
0\to M^{-1}  (-E) \to M^{-1} \to M^{-1} \vert_E \to 0.
\end{equation}
Again, since $\cO_X(-E)=L \otimes M^{-1}$, we get
\[
M^{-1} (-E) = L\otimes M^{-2}.
\]
Hence
\begin{equation}\label{eq:bundle}
\begin{split}
5 - 2s & = \chi(\Sigma,\cE \otimes H_\Sigma^{-1}) = \chi(E,M^{-1} \vert_E) \\
& = \chi(X,M^{-1}) - \chi(X,L\otimes M^{-2}) = \chi(X,M^{-1}) = 3,
\end{split}
\end{equation}
and thus $s=1$.

To finish the proof, we only need to prove that $\cE$ is stable; indeed, it is then automatically spherical, since its Mukai vector has square  $-2$.
Since $\cE$ is a rank-2 vector bundle of slope $(H_\Sigma\cdot c_1(\cE))/(H_\Sigma^2)\rk(\cE)=1/2$ and $\Sigma$ has Picard number 1, it is enough to prove   \mbox{$H^0(\Sigma,\cE\otimes H_\Sigma^{-1})=0$.}\
This is a slight refinement of~\eqref{eq:bundle} above.\
Indeed,
\[
H^0(\Sigma,\cE\otimes H_\Sigma^{-1})=H^0(E,(L\otimes (L\otimes M)^{-1})\vert_E)=H^0(E,M^{-1} \vert_E).
\]
From~\eqref{eq:Paris20211028}, we obtain an exact sequence
\[
H^0(X,M^{-1}) \to H^0(E,M^{-1} \vert_E) \to H^1(X,L\otimes M^{-2}).
\]
Under our assumptions, $\mm$ is contained in the closure of the positive cone and therefore $-\mm$ cannot be effective; hence, $H^0(X,M^{-1})=0$.\
Therefore, we only have to show  $H^1(X,L\otimes M^{-2})=0$ or, by Serre duality,  $H^3(X,  L^{-1}\ot M^{\otp 2})=0$.
This follows immediately from Lemma~\ref{lem:vanishing}.
\end{proof}

Next, we use Lemma~\ref{lem:bundle} to show $k_L=1$ and study the induced map $f_E\coloneqq f \vert_E \colon E\to \P^2$.

\begin{lemm}\label{lem:k=1}
The restriction morphism
\[
H^0(X,L)\lra H^0(E, L \vert_E) \isom \C^3
\]
is an isomorphism.
In particular, $k_L=1$ and $L$ is globally generated.
\end{lemm}

\begin{proof}
We consider again the exact sequence~\eqref{eq:Paris20211029}.\
By Lemma~\ref{lem:vanishing} and Serre duality, we get
\[
h^i(X,L (-E)) = h^i(X,L^{\otp 2}\otimes M^{-1})=h^{4-i}(X, L^{-2}\otimes M)=0,
\]
for $i\in\{0,1\}$, as we wanted.
The equality $h^0(E, L \vert_E)=3$ and the last statement follow since, by the Riemann--Roch Theorem, we have that $h^0(E, L \vert_E)\geq 3$, and then we apply Lemma~\ref{lem:H0large}.
\end{proof}

The moduli space $\M_0(\Sigma)$ was defined in Example~\ref{ex:K3n}.\
Let us denote by $E_0$ the exceptional divisor given by the universal family of curves in $\Sigma$ over $\P^2=|\hh |$.
Then $E_0$ is isomorphic to $ E$ over $\Sigma$ since they are projective bundles associated with  the same vector bundle $\cE$.

To finish the proof of Theorem~\ref{thm:MainThm} and Theorem~\ref{thm:SYZ}, we only need to show the following result, which can be seen as a version of~\cite[Theorem 1]{nag}.

\begin{prop}\label{thm:OGrady}
Let $(X,\lll,\mm)$ be a very general triple satisfying~\eqref{hypo} with $\mm$ isotropic and $\lll$ nef.\
Then $X$ is isomorphic to $\M_0(\Sigma)$.\
In particular, $X$ is of $\mathrm{K3}^{[2]}$ deformation type.
\end{prop}

\begin{proof}
By Lemma~\ref{lem:k=1}, the morphism $f_E$ coincides with the composition
\[
E \isomlra E_0 \lhra \M_0(\Sigma) \xrightarrow{\ f_0\ } \P^2,
\]
and thus $E$ is isomorphic to the universal family of curves in $\Sigma$ over $\P^2=|\hh |$.
Let us consider the nonempty open subset $U\subset\P^2$ where all morphisms $f$, $f_0$ and their restrictions respectively to $E$ and $E_0$ are smooth.
Then the restriction $f_E \vert_{U_E}$ of $f_E$ to $U_E\coloneqq f_E^{-1}(U)$ will factor via the relative Albanese variety $\Alb (U_E/U)$ (see~\cite[Proposition 1]{fujiki}): there exists a proper morphism
\[
g\colon \Alb (U_E/U) \lra f^{-1}(U)
\]
over $U$ compatible with the inclusion of $U_E$.

For all points $u\in U$, the class of the curve $C_u\coloneqq f_E^{-1}(u)$ gives a principal polarization on the abelian surface $A_u\coloneqq f^{-1}(u)$.
Since $C_u$ lives in a very general K3 surface, $A_u$ cannot be isomorphic to the product of two elliptic curves.\
Hence, $A_u$ is isomorphic to the Jacobian of~$C_u$ and the principal polarization is the theta polarization.\
From this, we immediately deduce that $g$ is an isomorphism.\
But the relative Albanese variety only depends on $f_E$ and thus it is isomorphic to the relative Albanese variety of the universal family of curves in $\Sigma$ over $\P^2=|\hh |$.\
This gives a birational morphism $X \isomdra  \M_0(\Sigma)$ which is then an isomorphism, since the nef and movable cones coincide.
\end{proof}


\section{Further results}\label{sec:further}
We keep the same setup: $X$ is a \hk\ manifold of dimension $2n$ with classes $\lll,\mm\in H^2(X,\Z)$ such that $\int_X\lll^{2n}=0$ and
$$a\coloneqq \frac 1{n!} \int_X\lll^n\mm^n$$
is a nonzero integer (Lemma~\ref{lemmaa});  changing $\mm$ into $-\mm$ if necessary, we may assume that the integers~$q_X(\lll,\mm)$, hence also $a$, are positive.\ After dealing for most of   this article  with the case~$a=1$, we examine in this section the general case.

\subsection{Boundedness results}

We first prove a general boundedness result, assuming $n$ and $a$ are fixed (this is~\cite[Theorem~4.9]{kam}).

\begin{prop}\label{prop41}
Let $a$ be a fixed positive integer.\
The number of deformation types of \hk\ manifolds $X$ of fixed dimension $2n$ for which there are classes~$\lll,\mm\in H^2(X,\Z)$ such that $\int_X \lll^{2n}=0$ and $\int_X \lll^{n}\mm^{n}=an!$ is finite.
\end{prop}

\begin{proof}
As explained above, we may assume $q_X(\lll,\mm)>0$.\
Since we can always add to $\mm$ a  multiple of $\lll$ (which changes $q_X(\mm)$ by adding multiples of $2q_X(\lll,\mm)$ but neither $q_X(\lll,\mm)$ nor~$a$), we may further assume
\begin{equation}\label{eqq}
-q_X(\lll,\mm)<q_X(\mm)\le q_X(\lll,\mm).
\end{equation}

We have, using~\eqref{fuj}, \eqref{eqq}, and \eqref{defia},
\begin{equation*}
\int_X(\lll+\mm)^{2n}=c_X q_X(\lll+\mm)^n =c_X(2 q_X(\lll,\mm)+q_X(\mm) )^n\le c_X 3^n q_X(\lll,\mm)^n= a 3^n\frac{(2n)!}{2^nn!}  .
\end{equation*}
By the surjectivity of the period map, there exists a deformation of $X$ whose N\'eron--Severi group is generated by the class $\lll+\mm$, which is therefore ample (or antiample).\
One can then apply~\cite[Corollary~1.2]{huyf} to conclude.
\end{proof}

 In dimension~$4$, we improve this result by allowing $a$ to take infinitely many values.\

\begin{prop}\label{prop42}
The number of deformation types of  \hk\ fourfolds $X$   for which there are classes~$\lll, \mm\in H^2(X,\Z)$ such that   \mbox{$\int_X \lll^4=0$} and  the   integer $a= \frac{1}{2}\int_X \lll^{2}\mm^{2}$ is positive and square-free, is finite.\end{prop}

\begin{proof}
The proposition is a consequence of Lemma~\ref{lemmax}, whose notation we keep.\ Indeed, that lemma and~\eqref{ax}
 say that the integer $aA'_X\coloneqq 12^2 \cdot 2aA_X$ is a perfect (positive) square.\
Since~$a$ is square-free and $A'_X\le 262$ by Lemma~\ref{lemmguan}, the integer $a$ is a product of (distinct) prime numbers that are all $<262$.\
It is therefore bounded (by the product of all these primes) and boundedness for $X$ follows from Proposition~\ref{prop41}.
\end{proof}

\subsection{Low values of $a$ for \hk\ fourfolds}

We obtain restrictions on the values that the integer $a$ may take for \hk\ fourfolds.\ The case $a=1$ was analyzed in Theorem~\ref{th43a} and is completely clarified by Theorem~\ref{thm:MainThm}:   it is only realized by \hk\ manifolds of~$\mathrm{K3}^{[2]}$ deformation type.

\begin{theo}\label{th43}
Let $X$ be a  \hk\ fourfold with classes~$\lll,\mm\in H^2(X,\Z)$ such that \mbox{$\int_X \lll^4=0$}.\   Assume $a\coloneqq \frac{1}{2}\int_X \lll^{2}\mm^{2}\in \{2,\dots,8\}$.\
 We are in one of the following cases:
\begin{enumerate}[{\rm (a)}]
 \item\label{enum:th43a} either $a=3$,
 $q_X(\lll,\mm)=1$, $c_X=9$,     $P_{RR,X}(T)= 3\binom{\frac{T}2+2}{2}$, and the quadratic form $q_X$ is even; moreover,
\subitem-- either  $(b_2(X),b_3(X),b_4(X))=(7,8, 108)${\rm ;}
 \subitem-- or $(b_2(X),b_3(X),b_4(X))=(6,4, 102)${\rm ;}
\subitem-- or $(b_2(X),b_3(X),b_4(X))=(5,0, 96)$.
\item\label{enum:th43b} or $a=4$, the Chern and Hodge numbers of $X$ are those of the Hilbert square of a K3 surface, and
\subitem-- either $q_X(\lll,\mm)=2$, $c_X=3$, the form $q_X$ is even, and $P_{RR,X}(T)=\binom{\frac{T}2+3}{2}${\rm ;}
\subitem-- or $q_X(\lll,\mm)=1$, $c_X=12$, and $P_{RR,X}(T)=\binom{T+3}{2}$.
 \end{enumerate}
\end{theo}

\begin{proof}
We follow the proof of Theorem~\ref{th43a}, which dealt with the case $a=1$:  we may assume  $\gamma\coloneqq \frac{q_X(\mm)}{q_X(\lll,\mm)}\in (-1,1] $, we introduce the polynomial
$$P(k)\coloneqq P_{RR,X}(q_X(k\lll+ \mm))=P_{RR,X}(2kq_X(\lll,\mm)+q_X(\mm)),$$
and we compute
\begin{align*}
 P(k)&=\frac{a}{2}k^2+\Bigl( \frac{a}{2}\gamma+2\sqrt{2aA_X} \Bigr)k+\frac{a}{8}\gamma^2+ \gamma\sqrt{2aA_X}+3\\
 &=:\frac{a}{2}k^2 +  b  k +c.
\end{align*}
Since $P$ takes integral values on integers, $\frac{a}2+b=P(1)-P(0)$ and $c=P(0)$ are integers; when~$a$ is odd, we write $b=\frac12+b'$, with $b'\in\Z$.\ We also check
 \begin{equation}\label{abc}
    4A_X-\frac{b^2}{2a} =3-c \in\Z.
\end{equation}
 Finally, when $a$ is not a perfect square, the fact that $\sqrt{2aA_X}$ is rational (Lemma~\ref{lemmax}), implies $A_X\ne\frac{25}{32}$ hence, by Lemma~\ref{lemmguan}\ref{bax},    $\frac56\le A_X\le\frac{131}{144}$.

\noindent{\bf Case $a=2$.} By~\eqref{abc}, we  have $4A_X-\frac{b^2}4\in\Z$ and $b\in\Z$.\ By Lemma~\ref{lemmguan}, this case does not happen.

\medskip
\noindent{\bf Case $a=3$.} By~\eqref{abc}, we  have $4A_X- \frac{b^2}6=4A_X-\frac{b'(b'+1)}6-\frac1{24}\in\Z$, hence $4A_X -\frac1{24}\in \Z+\{0,\frac13\}$.\
Since $a$ is not a perfect square, we have (as noted above) $\frac{5}{6}\le A_X\le\frac{131}{144} $, hence
\[
\frac{10}{3} -\frac1{24} \le 4A_X -\frac1{24}\le\frac{131}{36} -\frac1{24}.
\]
The only possibility is  $4A_X-\frac1{24}=3+\frac{1}3$, that is, $A_X=\frac{27}{32}$.\
This implies
$$\frac12+b'=b=\frac{3}{2}\gamma+\frac92,$$
so that $\gamma$ is an even integer.\ As above, this implies $\gamma=q_X(\mm)=0$, $b=\frac92$, and $c = 3$.\ Furthermore,
\[
P_{RR,X}(2kq_X(\lll,\mm))=  \frac{3}{2}k^2 +  \frac92  k +3=3\binom{k+2}{2}.
\]
By Lemma~\ref{star} (applied with $c=2$, $c'=3$, and $q=2q_X(\lll,\mm)$),
we get $q_X(\lll,\mm)=1$, the form~$q_X$ is even, $c_X=9$, and $P_{RR,X}(T)=   3\binom{\frac{T}2+2}{2}$, as in the generalized Kummer case.

From~\eqref{ax}, we obtain $c_4(X)=108$ and $4b_2(X)-b_3(X)=20 $.\
The possible Betti numbers listed in the theorem then follow from \cite{gua}.

\medskip
\noindent{\bf Case $a=4$.} By~\eqref{abc}, we   have $4A_X-\frac{b^2}8\in\Z$ and $b\in\Z$.\
If $\frac{5}{6}\le A_X\le\frac{131}{144} $, we have $b\equiv 2\pmod4$ and $A_X=\frac78$, which contradicts the fact that $\sqrt{2aA_X}$ is   rational.\
Hence, by Lemma~\ref{lemmguan}, we have $A_X=\frac{25}{32}$ and $b^2\equiv 1\pmod8$, so that $b$ is odd.\ Furthermore, $b=2\gamma +5$, hence $\gamma$ is an integer, which can only be $0$ or $1$.

When $\gamma=q_X( \mm)=0$ and $b=5$, we
obtain $c=3$ and
\[
P_{RR,X}(2kq_X(\lll,\mm))=P(k)= 2k^2 + 5  k +3=\binom{2k+3}{2}.
\]
When
$\gamma=1$ and $b=7$, we   get $q_X(\lll,\mm)=q_X(\mm)$ and $c=6$, and
\[
P_{RR,X}((2k+1)q_X(\lll,\mm))=P(k)= 2k^2 +7  k +6=\binom{(2k+1)+3}{2}.
\]
 In both cases, we have
 $$
 P_{RR,X}(q_X(\lll,\mm)T)=\binom{T+3}{2}.
 $$
 By Lemma~\ref{star} (applied with $c=2$, $c'=1$, and $q=q_X(\lll,\mm)$),
we get $q_X(\lll,\mm)\in\{1,2\}$.

Finally, since $A_X=\frac{25}{32}$, Lemma~\ref{lemmguan} implies that  the Chern and Hodge numbers of $X$ are those of the Hilbert square of a K3 surface.

\medskip
\noindent{\bf Case $a=5$.} We   obtain $4A_X-\frac{b^2}{10}\in\Z$ and $b\in  \Z +\frac12$ and, since $a$ is not a perfect square, $\frac{5}{6}\le A_X<1 $.\
We get $A_X=\frac{29}{32} $, which contradicts  $\sqrt{2aA_X}\in\Q$.

\medskip
\noindent{\bf Case $a=6$.} We   obtain $4A_X-\frac{b^2}{12}\in\Z$ and $A_X\in\{\frac56,\frac{15}{16}\}$, which contradicts  $\sqrt{2aA_X}\in\Q$.

\medskip
\noindent{\bf Case $a=7$.} We   obtain $4A_X-\frac{b^2}{14}\in\Z$ and $A_X\in\{\frac{193}{224},\frac{217}{224}\}$, which contradicts  $\sqrt{2aA_X}\in\Q$.

\medskip
\noindent{\bf Case $a=8$.} We   obtain $4A_X-\frac{b^2}{16}\in\Z$ and $A_X=\frac{57}{64}$, which contradicts  $\sqrt{2aA_X}\in\Q$.
\end{proof}

As we saw in Example~\ref{ex:Kumn},  the first item of the case $a=3$ in Theorem~\ref{th43} is realized by \hk\ fourfolds  of generalized Kummer deformation type;  we do not know whether the other two items occur.\
In the case $a=4$, the first item is realized  on \hk\ fourfolds of $\mathrm{K3}^{[2]}$ deformation type with $\mm$ divisible by 2.

Under a condition on some generalized Fujiki constants of $X$, Beckmann and Song prove in~\cite{bs} (see also~\cite{saw2}) that the only possible Betti numbers $b_2(X)$, $b_3(X)$, and $b_4(X)$ for $X$ are as in case~\ref{enum:th43a} of Theorem~\ref{th43}.\ 
Moreover, by~\cite[Proposition 5.6]{bs}, \cite[Conjecture 1.2]{bs} would imply that the Fujiki constant $c_X$ is either~$3$ or~$9$.\ In particular, the second case in~\ref{enum:th43b} should not occur.


\end{document}